\documentclass{article}
\usepackage{amssymb,bm}
\usepackage{amsmath}
\usepackage{amsthm}
\usepackage{graphicx}
\usepackage[active]{srcltx} 
\usepackage{hyperref}
\hypersetup{pdfborder=0 0 0}

\newtheorem{theorem}{{\sc Theorem}}[section]

\newtheorem{lemma}[theorem]{{\sc Lemma}}
\newtheorem{corollary}[theorem]{Corollary}
\newtheorem{remark}[theorem]{Remark}

\newtheorem{definition}[theorem]{Definition}

\newcommand{\bb}[1]{\mathbb{ #1}}

\newcommand{\dOm}{\partial\Omega}

\newcommand{\bra}[1]{\overline{#1}}

\newcommand{\Trc}{\mathrm{Tr}\,}

\newcommand{\tns}[1]{#1\otimes #1}
\newcommand{\hf}{\displaystyle\frac{1}{2}}
\newcommand{\nth}[1]{\displaystyle\frac{1}{#1}}

\newcommand{\dif}[2]{\displaystyle\frac{\partial #1}{\partial #2}}
\newcommand{\Grad}{\nabla}
\newcommand{\Div}{\nabla \cdot}

\renewcommand{\Hat}[1]{\widehat{#1}}
\newcommand{\Tld}[1]{\widetilde{#1}}

\newcommand{\Av}[1]{\skew{33}-{\int_{#1}}}

\newcommand{\re}{\Re\mathfrak{e}}

\newcommand{\mat}[4]{\left[\begin{array}{cc}
\displaystyle{#1}&\displaystyle{#2}\\[1ex]
\displaystyle{#3}&\displaystyle{#4}\end{array}\right]}

\newcommand{\lims}{\mathop{\overline\lim}}
\newcommand{\limi}{\mathop{\underline\lim}}

\newcommand{\bc}{boundary condition}

\newcommand{\rhs}{right-hand side}

\newcommand{\IFF}{if and only if }

\newcommand{\Ga}{\alpha}
\newcommand{\Gb}{\beta}
\newcommand{\Gd}{\delta}
\newcommand{\Ge}{\epsilon}

\newcommand{\Gvf}{\varphi}
\newcommand{\Gg}{\gamma}

\newcommand{\Gl}{\lambda}

\newcommand{\Gth}{\theta}

\newcommand{\Gs}{\sigma}

\newcommand{\GD}{\Delta}

\newcommand{\GG}{\Gamma}

\newcommand{\GL}{\Lambda}

\newcommand{\GO}{\Omega}

\bmdefine\BGa{\alpha}
\bmdefine\BGb{\beta}
\bmdefine\BGd{\delta}
\bmdefine\BGe{\epsilon}
\bmdefine\BGve{\varepsilon}
\bmdefine\BGf{\phi}
\bmdefine\BGvf{\varphi}
\bmdefine\BGg{\gamma}
\bmdefine\BGc{\chi}
\bmdefine\BGi{\iota}
\bmdefine\BGk{\kappa}
\bmdefine\BGl{\lambda}
\bmdefine\BGn{\eta}
\bmdefine\BGm{\mu}
\bmdefine\BGv{\nu}
\bmdefine\BGp{\pi}
\bmdefine\BGth{\theta}
\bmdefine\BGvth{\vartheta}
\bmdefine\BGr{\rho}
\bmdefine\BGvr{\varrho}
\bmdefine\BGs{\sigma}
\bmdefine\BGvs{\varsigma}
\bmdefine\BGt{\tau}
\bmdefine\BGj{\tau}
\bmdefine\BGu{\upsilon}
\bmdefine\BGo{\omega}
\bmdefine\BGx{\xi}
\bmdefine\BGy{\psi}
\bmdefine\BGz{\zeta}
\bmdefine\BGD{\Delta}
\bmdefine\BGF{\Phi}
\bmdefine\BGG{\Gamma}
\bmdefine\BGL{\Lambda}
\bmdefine\BGP{\Pi}
\bmdefine\BGT{\Theta}
\bmdefine\BGS{\Sigma}
\bmdefine\BGU{\Upsilon}
\bmdefine\BGO{\Omega}
\bmdefine\BGX{\Xi}
\bmdefine\BGY{\Psi}

\newcommand{\CA}{{\mathcal A}}
\newcommand{\CB}{{\mathcal B}}
\newcommand{\CC}{{\mathcal C}}

\newcommand{\CE}{{\mathcal E}}
\newcommand{\CF}{{\mathcal F}}

\newcommand{\CL}{{\mathcal L}}

\newcommand{\CP}{{\mathcal P}}

\newcommand{\CU}{{\mathcal U}}

\bmdefine\BCA{{\mathcal A}}
\bmdefine\BCB{{\mathcal B}}
\bmdefine\BCC{{\mathcal C}}
\bmdefine\BCD{{\mathcal D}}
\bmdefine\BCE{{\mathcal E}}
\bmdefine\BCF{{\mathcal F}}
\bmdefine\BCG{{\mathcal G}}
\bmdefine\BCH{{\mathcal H}}
\bmdefine\BCI{{\mathcal I}}
\bmdefine\BCJ{{\mathcal J}}
\bmdefine\BCK{{\mathcal K}}
\bmdefine\BCL{{\mathcal L}}
\bmdefine\BCM{{\mathcal M}}
\bmdefine\BCN{{\mathcal N}}
\bmdefine\BCO{{\mathcal O}}
\bmdefine\BCP{{\mathcal P}}
\bmdefine\BCQ{{\mathcal Q}}
\bmdefine\BCR{{\mathcal R}}
\bmdefine\BCS{{\mathcal S}}
\bmdefine\BCT{{\mathcal T}}
\bmdefine\BCU{{\mathcal U}}
\bmdefine\BCV{{\mathcal V}}
\bmdefine\BCW{{\mathcal W}}
\bmdefine\BCX{{\mathcal X}}
\bmdefine\BCY{{\mathcal Y}}
\bmdefine\BCZ{{\mathcal Z}}

\bmdefine\Bzr{ 0}
\bmdefine\Ba{ a}
\bmdefine\Bb{ b}
\bmdefine\Bc{ c}
\bmdefine\Bd{ d}
\bmdefine\Be{ e}
\bmdefine\Bf{ f}
\bmdefine\Bg{ g}
\bmdefine\Bh{ h}
\bmdefine\Bi{ i}
\bmdefine\Bj{ j}
\bmdefine\Bk{ k}
\bmdefine\Bl{ l}
\bmdefine\Bm{ m}
\bmdefine\Bn{ n}
\bmdefine\Bo{ o}
\bmdefine\Bp{ p}
\bmdefine\Bq{ q}
\bmdefine\Br{ r}
\bmdefine\Bs{ s}
\bmdefine\Bt{ t}
\bmdefine\Bu{ u}
\bmdefine\Bv{ v}
\bmdefine\Bw{ w}
\bmdefine\Bx{ x}
\bmdefine\By{ y}
\bmdefine\Bz{ z}
\bmdefine\BA{ A}
\bmdefine\BB{ B}
\bmdefine\BC{ C}
\bmdefine\BD{ D}
\bmdefine\BE{ E}
\bmdefine\BF{ F}
\bmdefine\BG{ G}
\bmdefine\BH{ H}
\bmdefine\BI{ I}
\bmdefine\BJ{ J}
\bmdefine\BK{ K}
\bmdefine\BL{ L}
\bmdefine\BM{ M}
\bmdefine\BN{ N}
\bmdefine\BO{ O}
\bmdefine\BP{ P}
\bmdefine\BQ{ Q}
\bmdefine\BR{ R}
\bmdefine\BS{ S}
\bmdefine\BT{ T}
\bmdefine\BU{ U}
\bmdefine\BV{ V}
\bmdefine\BW{ W}
\bmdefine\BX{ X}
\bmdefine\BY{ Y}
\bmdefine\BZ{ Z}

\newcommand{\SFL}{\mathsf{L}}

\pdfoutput=1
\title{Rigorous asymptotic analysis of buckling of thin-walled cylinders under
  axial compression}
\author{Yury Grabovsky \and Davit Harutyunyan}
\newcommand{\ud}{\,\mathrm{d}}
\begin{document}
\maketitle
\begin{abstract}
  Using rigorous constitutive linearization of second variation introduced in
  \cite{grtr07} we study weak
  stability of homogeneous deformation of the axially compressed circular cylindrical
  shell, regarded as a 3-dimensional hyperelastic body. We show that such
  deformation becomes weakly unstable at the citical load that coincides with
  value of the bifurcation load in von-K\'arm\'an-Donnel shell
  theory. We also show that the linear bifurcation modes described by the
  Koiter circle \cite{koit45} minimize the second variation asymptotically. The key
  ingredients of our analysis are the asymptoticaly sharp estimates of the Korn constant for
  cylindrical shells and Korn-like inequalities on components of the
  deformation gradient tensor in cylindrical coordinates. The notion of buckling equivalence
  introduced in \cite{grtr07} is developed further and becomes central in this
  work. A link between features of this theory and sensitivity of the critical
  load to imprefections of load and shape is conjectured. 
\end{abstract}


\section{Introduction}
\label{sec:intro}
Recent years have seen significant progress in rigorous analysis of
dimensionally reduced theories of plates and shells based on $\GG$-convergence
\cite{fjm02,momu04,fjm06,lmap2011,lmop2011}. In the framework of these
theories one must postulate the scaling of energy and the forces apriori,
whereby different scaling assumptions lead to different dimensionally reduced
plate and shell theories. By contrast \cite{grtr07} has no need for such
apriori assumptions, while pursuing a less general goal of identifying a
critical load at which the trivial branch of equilibria becomes weakly
unstable. This exclusive targeting of the instability without the attempt to capture
post-buckling behavior leads to significant technical simplifications in a
rigorous analysis of the safe load problem for slender structures. Present
work builds on the ideas of \cite{grtr07} and applies them to buckling of a
circular cylindrical shell under axial compression.

The asymptotics of the critical buckling load predicted by the sign of the
3-dimensional second variation agrees with the classical value in the shell
theory \cite{lorenz11,timosh14,koit45}. The
displacement variations that minimize the normalized second variation of the
energy are single Fourier modes, whose wave numbers lie on Koiter's circle
\cite{koit45}. Using the notion of buckling equivalence \cite{grtr07} we link
the classical variational problem for the buckling load obtained from the
shell theory and the rigorous analysis of the sign of the 3-dimensional
second variation of the non-linear elastic energy.

The problem of buckling of axially compressed cylindrical shells occupies a
special place in engineering. Cylindrical shells are light weight structures
with superior load carrying capacity in axial direction as compared to plates
of the same thickness. They are ubiquitous in industry. Yet, the classical
theoretical value of the buckling load is about 4 to 5 times higher than the
one observed in experiments \cite{call2000}. This is understood as a
manifestation of the sensitivity of the buckling load to imperfections of load
and shape \cite{almr63,tenn64,wms65,goev70,yama84}. The general understanding
of the sensitivity of the buckling load to imperfections is via the 
bifurcation theory applied to von-K\'arm\'an-Donnell equations
\cite{hune91,hune93,lch97,hlc99}. The subcritical nature of the bifurcation
\cite{hune91,hlc99}, where the load drops sharply at the critical load is
responsible for the observed discrepancy between the theoretical and apparent
value of the critical load. It is important to note that the formal
asymptotics leading to the theoretical value of the critical load gives no
indication of the nature of bifurcation and the resultant imperfection sensitivity.

Our approach to buckling of a circular cylindrical shell, along with the
rigorous derivation of the classical buckling load, may offer an alternative
interpretation of the sensitivity of the critical load to imperfections. Our
analysis makes apparent three different asymptotics of the buckling load, only
one of which (the one with the largest critical load) is realized in a perfect
uniformly axially compressed cylindrical shell. Imprefections of load and
shape lead to small perturbations in a trivial branch that are shown to be
sufficient to change the asymptotics of the buckling load. While the rigorous analysis
requires more work and is left to future studies, we conjecture that the two
other buckling loads that are significantly lower that the classical one could
lead to a more transparent explanation of imperfection sensitivity.

Recall that a stable configuration $\By=\By(\Bx)$,
$\Bx\in\GO\subset\bb{R}^{3}$ must be a weak local minimizer of the energy
\[
\int_{\GO}W(\Grad\By)d\Bx-\int_{\dOm}(\By,\Bt(\Bx))dS(\Bx),
\]
where $W(\BF)$ is the energy density function of the body and $\Bt(\Bx)$ is
the vector of dead load tractions. The essential non-linearity of buckling
comes from the principle of frame indifference $W(\BR\BF)=W(\BF)$ for all
$\BR\in SO(3)$, combined with the assumption of the absence of residual stress
$W_{\BF}(\BI)=\Bzr$. For slender bodies these assumptions are fundamental for
the computation of the constitutive linearization, which is based on the fact
that the stresses are small at every point in the body right up to the
buckling point, and therefore, the material response can be linearized
\emph{locally}. The impossibility of the geometric linearization due to the
distributed nature of local rotations is the essential feature of slender
bodies.

The constitutively linearized problem can be viewed as the variational
formulation of the linear eigenvalue problem at the bifurcation point. As such
it permits an extra degree of flexibility, as one can replace one variational
formulation with an asymptotically equivalent one \cite{grtr07}. This
flexibility is used in this paper to simplify and eliminate heavy algebraic
calculations for the perfect cylinder.

This paper is organized as follows. In Section~\ref{sec:genth} we extend the
general theory of buckling developed in \cite{grtr07}, so that it applies to
more general 3-dimensional bodies, including cylindrical shells. We define an
equivalence class of functionals characterizing buckling and give criteria for
showing that a pair of functionals belongs to the same equivalence class. In
Section~\ref{sec:KI} we discuss the asymptotics of the Korn constant for
cylindrical shells. The technical details of the proof are in
Appendix~\ref{app:KI}. In Section~\ref{sec:KLI} we prove Korn-like
inequalities where the linear strain is bounded in terms of the specific
components if the displacement gradient. In Section~\ref{sec:trbr} we compute
the compressive part of the constitutively linearized functional whose
destabilizing action is ultimately responsible for buckling when the load
reaches critical.  In Section~\ref{sec:perfect} we apply the general theory of
buckling from Section~\ref{sec:genth} to axially compressed perfect
cylindrical shells and derive the formula for the buckling load, as well as a
collection of buckling modes parametrized by points on the Koiter's circle
\cite{koit45}. In Section~\ref{sec:altbc} we show that the more realistic but
also more technically challenging \bc s, where displacements are prescribed on
the top and bottom boundaries of the shell produce exactly the same buckling
load. This is achieved by exploiting the massive non-uniqueness of the
buckling modes for the the prescribed average vertical displacement \bc s of
Section~\ref{sec:perfect}. In Section~\ref{sec:conc} we comment on the
possibility of the link between the sensitivity to imperfection of the
buckling load of a slender structure and the presence of ``latent'' buckling
modes with significantly smaller critical loads.

\section{Buckling of slender structures}
\setcounter{equation}{0}
\label{sec:genth}
Here we revisit the general theory of buckling developed in \cite{grtr07}. The
theory deals with a sequence of progressively slender domains $\GO_{h}$
parametrized by a dimensionless parameter $h$. For example, in the case of the
cylindrical shell, $h$ is the ratio of cylinder wall thickness to the cylinder
radius. We consider a loading program parametrized by the loading parameter
$\Gl$ describing the magnitude of the applied tractions
$\Bt(\Bx;h,\Gl)=\Gl\Bt^{h}(\Bx)+O(\Gl^{2})$, as $\Gl\to 0$. Here and below
$O(\cdot)$ is understood uniformly in $\Bx\in\GO_{h}$ and $h\in[0,h_{0}]$. Two
fundamental assumptions need to be made in order for the general theory of
buckling to be applicable\footnote{Some relaxation of the uniformity
  in $h$ assumption might be necessary in order to bring our theory to bear on
the question of sensitivity of the buckling load to imperfections}.

The first fundamental assumption requires the \emph{existence} of the family of
equilibrium deformations $\By(\Bx;h,\Gl)$, corresponding to the applied loads
$\Bt(\Bx;h,\Gl)$ and satisfying the imposed \bc s, for any $h\in[0,h_{0}]$ and
$\Gl\in[0,\Gl_{0}]$, where $h_{0}>0$ and $\Gl_{0}>0$ are some constants.  Such
a family of equilibria will be called a ``trivial branch''. Neither uniqueness
nor its stability is assumed.

The second fundamental assumption is the absence of ``bending modes'' in the
trivial branch. Here we use the term ``bending'' loosely to indicate any
response in which the strain to stress ratio becomes infinitely large as $h\to
0$ even for small applied stress. Formally we assume that the trivial branch
lies uniformly close to the linearly elastic response:
\begin{equation}
  \label{fundass}
  \sup_{0\le h\le h_{0}}\|\Grad\By(\Bx;h,\Gl)-\BI-\Gl\Grad\Bu^{h}(\Bx)\|_{L^{\infty}(\GO_{h})}\le
C\Gl^{2},\qquad\Bu^{h}(\Bx)=\left.\dif{\By(\Bx;h,\Gl)}{\Gl}\right|_{\Gl=0},
\end{equation}
when $0\le\Gl\le\Gl_{0}$ and the constant $C$ is independent of $h$. 

In \cite{grtr07} we have defined the notion of the near-flip buckling when
for any $h\in[0,h_{0}]$ the trivial branch becomes unstable for $\Gl>\Gl(h)$,
where $\Gl(h)\to 0$, as $h\to 0$. This happens because it becomes
energetically more advantageous to activate bending modes rather than store
more compressive stress.

In hyperelasticity the trivial branch $\By(\Bx;h,\Gl)$ is a critical point of
the energy functional
\begin{equation}
  \label{energy}
\CE(\By)=\int_{\GO_{h}}W(\Grad\By)d\Bx-\int_{\dOm_{h}}(\Bt(\Bx;h,\Gl),\By)dS.
\end{equation}
In general we restrict $\By$ to an affine subspace of
$W^{1,\infty}(\GO_{h};\bb{R}^{3})$ given by
\begin{equation}
  \label{hdbc}
  \By\in\bra{\By}(\Bx;h,\Gl)+V_{h}^{\circ},
\end{equation}
where $V_{h}^{\circ}$ is a linear subspace of
$W^{1,\infty}(\GO_{h};\bb{R}^{3})$ that contains
$W_{0}^{1,\infty}(\GO_{h};\bb{R}^{3})$ and does not depend on the loading
parameter $\Gl$. The given function $\bra{\By}(\Bx;h,\Gl)\in
W^{1,\infty}(\GO_{h};\bb{R}^{3})$ describes the Dirichlet part of the \bc s,
while the traction vector $\Bt(\Bx;h,\Gl)$ describes the
Neumann-part\footnote{The use of a general subspace $V_{h}$ permits one to
  describe loadings in which desired linear combinations of the
  displacement and traction components are prescribed on the boundary.}.

The equilibrium equations and the \bc s satisfied by the trivial branch
$\By(\Bx;h,\Gl)$ can be written explicitly only in the weak form:
\begin{equation}
  \label{ELwk}
  \int_{\GO_{h}}(W_{\BF}(\Grad\By(\Bx;h,\Gl)),\Grad\BGf)d\Bx
-\int_{\dOm_{h}}(\Bt(\Bx;h,\Gl),\BGf)dS=0
\end{equation}
for every $\BGf\in V_{h}$, where $V_{h}$ is a closure of $V_{h}^{\circ}$ in
$W^{1,2}(\GO_{h};\bb{R}^{3})$.
Differentiating (\ref{ELwk}) in $\Gl$ at $\Gl=0$, which is allowed due to
(\ref{fundass}), we obtain
\begin{equation}
  \label{linel}
  \int_{\GO_{h}}(\SFL_{0}\Grad\Bu^{h}(\Bx),\Grad\BGf)d\Bx
-\int_{\dOm_{h}}(\Bt^{h}(\Bx),\BGf)dS=0,\qquad\BGf\in V_{h}
\end{equation}
where $\SFL_{0}=W_{\BF\BF}(\BI)$.

The energy density $W(\BF)$ satisfies two fundamental assumptions:
\begin{itemize}
\item[(P1)] Absence of prestress: $W_{\BF}(\BI)=\Bzr$;
\item[(P2)] Frame indifference: $W(\BF\BR)=W(\BF)$ for every $\BR\in SO(3)$;
\item[(P3)] Positive semidefiniteness property $(\SFL_{0}\BGx,\BGx)\ge 0$ for any
  $\BGx\in\bb{R}^{3\times 3}$;
\item[(P4)] Non-degeneracy $(\SFL_{0}\BGx,\BGx)=0$ \IFF $\BGx^{T}=-\BGx$.
\end{itemize}
By the properties (P3)--(P4) of $\SFL_{0}$, there exists $\Ga_{\SFL_{0}}>0$,
such that
\begin{equation}
  \label{Lcoerc}
  (\SFL_{0}\BGx,\BGx)\ge\Ga_{\SFL_{0}}|\BGx_{\rm sym}|^{2},\qquad\BGx_{\rm
    sym}=\hf(\BGx+\BGx^{T}).
\end{equation}

The buckling is detected by the second variation of energy
\[
\Gd^{2}\CE(\BGf;h,\Gl)=\int_{\GO_{h}}(W_{\BF\BF}(\Grad\By(\Bx;h,\Gl))\Grad\BGf,\Grad\BGf)d\Bx.
\]
The second variation is always non negative, when $0<\Gl<\Gl(h)$ and can become
negative for some choice of the admissible variation $\BGf\in V_{h}$, when
$\Gl>\Gl(h)$. It was understood in \cite{grtr07} that this failure of weak
stability is due to the properties (P1)--(P4) of $W(\BF)$ and is intimately
related to flip instability in soft device.  It was shown in \cite{grtr07}
that in an ``almost soft'' device the critical load could be captured (under
some assumptions) by the \emph{constitutively linearized} second variation
\begin{equation}
  \label{ccsv}
\Gd^{2}\CE_{cl}(\BGf;h,\Gl)=\int_{\GO_{h}}\{(\SFL_{0}e(\BGf),e(\BGf))-
\Gl(\BGs_{h},\Grad\BGf^{T}\Grad\BGf)\}d\Bx,
\end{equation}
where 
\begin{equation}
  \label{minusstress}
 \BGs_{h}(\Bx)=-\SFL_{0}e(\Bu^{h}(\Bx)) ,\qquad e(\Bu)=\hf(\Grad\Bu+(\Grad\Bu)^{T}).
\end{equation}
Observe, that $\BGs^{h}$ is \emph{minus} the linear stress in the
body. This notation is convenient when we are dealing exclusively with compressive
stresses. The larger the stress, the more compressive the load.
Let
\begin{equation}
  \label{Adef}
\CA_{h}=\left\{\BGf\in V_{h}:\int_{\GO_{h}}(\BGs_{h},\Grad\BGf^{T}\Grad\BGf)d\Bx>0\right\}
\end{equation}
be the set of all destabilizing variations (see (\ref{ccsv})).
\begin{definition}
  \label{def:wkcompr}
We say that the loading is \textbf{weakly compressive} if there exists
$h_{0}>0$ so that $\CA_{h}\not=\emptyset$, for all $h<h_{0}$.
\end{definition}
If the stress is not weakly compressive, then both terms in (\ref{ccsv})
are non-negative for all $\BGf$ and the instability is clearly impossible. If
the loading is weakly compressive then the two terms in (\ref{ccsv}) have
opposite signs for all variations $\BGf\in\CA_{h}$. The linearized second
variation can then become negative for sufficiently large $\Gl$. However, this
does not immediately imply negativity of the second variation
$\Gd^{2}\CE(\BGf;h,\Gl)$. To examine the sign of the second variation we
consider the function
\[
m^{*}(h,\Gl)=\inf_{\BGf\in\CA_{h}}\frac{\Gd^{2}\CE(\BGf;h,\Gl)}
{\int_{\GO_{h}}(\BGs_{h},\Grad\BGf^{T}\Grad\BGf)d\Bx}
\]
that measures reserve of stability in the
trivial branch. Negative values of $m^{*}(h,\Gl)$ signal instability. The
functional $m^{*}(h,\Gl)$ is based on the representation of the buckling load
as a generalized Korn constant in \cite{grtr07}.
\begin{remark}
  \label{rem:flip}
The theory of  buckling  of slender bodies in \cite{grtr07} is based on the
analysis of the function
\[
m(h,\Gl)=\inf_{\BGf\in V_{h}}\frac{\Gd^{2}\CE(\BGf;h,\Gl)}{\|\Grad\BGf\|^{2}}.
\]
We will see that for the axially compressed cylindrical shell
\begin{equation}
  \label{comparison}
  \int_{\GO_{h}}(\BGs_{h},\Grad\BGf_{h}^{T}\Grad\BGf_{h})d\Bx=o(\|\Grad\BGf_{h}\|^{2})
\end{equation}
for the minimizer $\BGf_{h}$ in $m^{*}(h,\Gl(h))$. Therefore,
\[
M(p)=\lim_{h\to 0}\frac{m(h,p\Gl(h))}{\Gl(h)}=0,
\]
and the link between $\Gd^{2}\CE(\BGf;h,\Gl)$ and
$\Gd^{2}\CE_{cl}(\BGf;h,\Gl)$ cannot be ascertained. In this paper we will
show that replacing $m(h,\Gl)$ with $m^{*}(h,\Gl)$ permits to establish the
required link. We also observe that the
distinction between $m(h,\Gl)$ and $m^{*}(h,\Gl)$ is essentially
3-dimensional, since, as we have shown in \cite{grtr07}, 
\[
\int_{\GO_{h}}(\BGs_{h},\Grad\BGf_{h}^{T}\Grad\BGf_{h})d\Bx\sim
\int_{\GO_{h}}\hf(\Trc\BGs_{h})|\Grad\BGf_{h}|^{2}d\Bx.
\]
\end{remark}
\begin{definition}
  \label{def:Bload}
We say that $\Gl(h)$ is the \textbf{buckling load} if
\[
\Gl(h)=\inf\{\Gl>0:m^{*}(h,\Gl)<0\}.
\]
We say that $\{\BGf_{h}\in\CA_{h}:h\in(0,h_{0})\}$ is a \textbf{buckling mode} if
\[
\lim_{h\to 0}\frac{\Gd^{2}\CE(\BGf_{h};h,\Gl(h))}
{\Gl(h)\int_{\GO_{h}}(\BGs_{h},\Grad\BGf_{h}^{T}\Grad\BGf_{h})d\Bx}=0,
\]
where $\Gl(h)$ is the buckling load.
\end{definition}
\begin{definition}
  \label{def:strcompr}
We will call the loading \textbf{strongly compressive}, if $\Gl(h)\to 0$, as $h\to 0$.
\end{definition}
According to (\ref{ccsv}), strongly compressive loads imply existence of variations
$\BGf\in\CA_{h}$ for which the measure of compressiveness 
\[
\mathfrak{C}_{h}(\BGf)=\int_{\GO_{h}}(\BGs_{h},\Grad\BGf^{T}\Grad\BGf)d\Bx
\]
is much larger than the measure of stability
\[
\mathfrak{S}_{h}(\BGf)=\int_{\GO_{h}}(\SFL_{0}e(\BGf),e(\BGf))d\Bx.
\]
In particular, the body $\GO_{h}$ has to be slender in the sense that the Korn
constant
\begin{equation}
  \label{KK}
  K(V_{h})=\inf_{\BGf\in V_{h}}\frac{\|e(\BGf)\|^{2}}{\|\Grad\BGf\|^{2}}
\end{equation}
is infinitesimal $K(V_{h})\to 0$, as $h\to 0$. 

The notion of the buckling mode in Definition~\ref{def:Bload} suggests an
extension of the notion of buckling equivalence to include buckling modes in
addition to buckling loads.
\begin{definition}
  \label{def:equiv}
  Assume $J(h,\BGf)$ is a variational functional defined on
  $\CB_h\subset\CA_{h}$. We say that the pair $(\CB_h, J(h,\BGf))$
  \textbf{characterizes buckling} if the following three conditions are
  satisfied
\begin{enumerate}
\item[(a)] Characterization of the buckling load:
\[
\lim_{h\to 0}\frac{\Hat{\Gl}(h)}{\Gl(h)}=1,
\]
where $\Gl(h)$ is the buckling load and
\[
\Hat{\Gl}(h)=\inf_{\BGf\in\CB_{h}}J(h,\BGf).
\]
\item[(b)] Characterization of the buckling mode:
If $\BGf_{h}\in\CB_{h}$ is a buckling mode, then
  \begin{equation}
    \label{bmode}
\lim_{h\to 0}\frac{J(h,\BGf_{h})}{\Hat{\Gl}(h)}=1.
  \end{equation}
\item[(c)] Faithful representation of the buckling mode:
If $\BGf_{h}\in\CB_{h}$ satisfies (\ref{bmode}) then it is a buckling mode.
\end{enumerate}
\end{definition}
We remark that by Definition~\ref{def:Bload} of the buckling mode the pair
$(\CA_{h},\mathfrak{J}(h,\BGf))$ characterizes buckling, where
\[
\mathfrak{J}(h,\BGf)=\Gl(h)+\frac{\Gd^{2}\CE(\BGf;h,\Gl(h))}{\mathfrak{C}_{h}(\BGf)}.
\]
We envision two ways in which the analysis of buckling can be simplified. One
is the simplification of the functional $\mathfrak{J}(h,\BGf)$. The other is
replacing the space of all admissible functions $\CA_{h}$ with a much smaller
space $\CB_{h}$. For example, we may want to use a specific ansatz, like the
Kirchhoff ansatz in buckling of rods and plates. According to
Definition~\ref{def:equiv} the simplified functional $J(h,\BGf)$, restricted to
the ansatz $\CB_{h}$ will capture the asymptotics of the buckling load and at
least one buckling mode. It is in principle conceivable, that there are other
buckling modes, not contained in $\CB_{h}$. However, we believe that such a
situation is non-generic. Even in this non-generic situation our approach
would allow to capture all geometrically distinct buckling modes, if one can
identify the ansatz $\CB_{h}$ for each of them.

It is an elementary observation that the only requirement we need to place on the
ansatz $\CB_{h}$ is that it must contain a buckling mode.
\begin{lemma}
\label{lem:pairB_hJ}
Suppose the pair $(\CB_{h},J(h,\BGf))$ characterizes buckling.
Let $\CC_{h}\subset\CB_{h}$ be such that it contains a buckling mode. Then the
pair $(\CC_{h},J(h,\BGf))$ also characterizes buckling.
\end{lemma}
\begin{proof}
Let
\[
\Hat{\Gl}(h)=\inf_{\BGf\in\CB_{h}}J(h,\BGf),\qquad\Tld{\Gl}(h)=\inf_{\BGf\in\CC_{h}}J(h,\BGf).
\]
Then, clearly, $\Tld{\Gl}(h)\ge\Hat{\Gl}(h)$. By assumption there exists a
buckling mode $\BGf_{h}\in\CC_{h}\subset\CB_{h}$. Therefore,
\[
\lims_{h\to 0}\frac{\Tld{\Gl}(h)}{\Hat{\Gl}(h)}\le
\lim_{h\to 0}\frac{J(h,\BGf_{h})}{\Hat{\Gl}(h)}=1,
\]
since the pair $(\CB_{h},J(h,\BGf))$ characterizes buckling. Hence
\begin{equation}
  \label{adeq}
 \lim_{h\to 0}\frac{\Tld{\Gl}(h)}{\Hat{\Gl}(h)}=1,
\end{equation}
and part (a) of Definition~\ref{def:equiv} is established.

If $\BGf_{h}\in\CC_{h}\subset\CB_{h}$ is a buckling mode then
\[
\lim_{h\to 0}\frac{J(h,\BGf_{h})}{\Hat{\Gl}(h)}=1,
\]
since the pair $(\CB_{h},J(h,\BGf))$ characterizes buckling. Part (b) now
follows from (\ref{adeq}).

Finally, if $\BGf_{h}\in\CC_{h}$ satisfies
\[
\lim_{h\to 0}\frac{J(h,\BGf_{h})}{\Tld{\Gl}(h)}=1,
\]
then, $\BGf_{h}\in\CB_{h}$ and by (\ref{adeq}) we also have
\[
\lim_{h\to 0}\frac{J(h,\BGf_{h})}{\Hat{\Gl}(h)}=1.
\]
Therefore, $\BGf_{h}$ is a buckling mode, since, by assumption the pair
$(\CB_{h},J(h,\BGf))$ characterizes buckling. The Lemma is proved now.
\end{proof}

If we replace the second variation $\Gd^{2}\CE(\BGf;h,\Gl)$ with the
constitutively linearized second variation (\ref{ccsv}) in the functional
$\mathfrak{J}(h,\BGf)$, we will obtain the functional
\begin{equation}
  \label{Kgen}
\mathfrak{K}(h,\BGf)=\frac{\int_{\GO_{h}}(\SFL_{0}e(\BGf),e(\BGf))d\Bx}
{\int_{\GO_{h}}(\BGs_{h},\Grad\BGf^{T}\Grad\BGf)d\Bx}=
\frac{\mathfrak{S}_{h}(\BGf)}{\mathfrak{C}_{h}(\BGf)},
\end{equation}
which first appeared in \cite{grtr07} where it was shown that the buckling
load can be regarded as a generalized Korn constant
\begin{equation}
  \label{clin}
\Hat{\Gl}(h)=\inf_{\BGf\in\CA_{h}}\mathfrak{K}(h,\BGf).
\end{equation}
Unfortunately the sufficient conditions for buckling equivalence established
in \cite{grtr07} fail to guarantee the validity of constitutive
linearization in the case of the axially compressed thin-walled cylinder for
reasons explained in Remark~\ref{rem:flip}. Our next theorem proves buckling
equivalence and also shows that the constitutively linearized functional
$\mathfrak{K}(h,\BGf)$ captures the buckling mode as well.
\begin{theorem}[Asymptotics of the critical load]
  \label{th:crit}
Suppose that the Korn constant $K(V_{h})$ defined by (\ref{KK}) satisfies
\begin{equation}
  \label{sufcond}
  \lim_{h\to 0}K(V_{h})=0,\qquad\lim_{h\to 0}\frac{\Hat{\Gl}(h)^{2}}{K(V_{h})}=0.
\end{equation}
Then the pair $(\CA_h, \mathfrak{K}(h,\BGf))$ characterizes buckling in the sense of
Definition~\ref{def:equiv}.
\end{theorem}
\begin{proof}
  The theorem is proved by means of the basic estimate, which is a simple
  modification of the estimates in \cite{grtr07} used in the derivation of the
  formula for $\Gd^{2}\CE_{cl}(\BGf;h,\Gl)$:
\begin{lemma}
  \label{leb:be}
Suppose $\By(\Bx;h,\Gl)$ satisfies (\ref{fundass}) and $W(\BF)$ has the
properties (P1)--(P2). Then
\begin{equation}
  \label{basic}
\left|\Gd^{2}\CE(\BGf;h,\Gl)-\Gd^{2}\CE_{cl}(\BGf;h,\Gl)\right|\le
C(\Gl\|e(\BGf)\|\|\Grad\BGf\|+\Gl^{2}\|\Grad\BGf\|^{2}).
\end{equation}
\end{lemma}
\begin{proof}
According to the frame indifference property (P2),
$W(\BF)=\Hat{W}(\BF^{T}\BF)$. Differentiating this formula twice we obtain
\[
(W_{\BF\BF}(\BF)\BGx,\BGx)=4(\Hat{W}_{\BC\BC}(\BC)(\BF^{T}\BGx),\BF^{T}\BGx)+
2(\Hat{W}_{\BC}(\BC)\BGx^{T}\BGx),\qquad\BC=\BF^{T}\BF.
\]
We make the following estimate
\[
|(\Hat{W}_{\BC\BC}(\BC)(\BF^{T}\BGx),\BF^{T}\BGx)-(\Hat{W}_{\BC\BC}(\BI)\BGx,\BGx)|\le
|(\Hat{W}_{\BC\BC}(\BC)(\BF^{T}-\BI)\BGx,(\BF^{T}-\BI)\BGx)|+
\]
\[
|((\Hat{W}_{\BC\BC}(\BC)-\Hat{W}_{\BC\BC}(\BI))\BGx,\BGx)|+
2|(\Hat{W}_{\BC\BC}(\BC)\BGx,(\BF^{T}-\BI)\BGx)|
\]
When $\BF$ is uniformly bounded we obtain
\[
|(\Hat{W}_{\BC\BC}(\BC)(\BF^{T}\BGx),\BF^{T}\BGx)-(\Hat{W}_{\BC\BC}(\BI)\BGx,\BGx)|\le
C\left(|\BF-\BI|^{2}|\BGx|^{2}+|\BC-\BI||\BGx_{\rm sym}|^{2}+|\BF-\BI||\BGx_{\rm sym}||\BGx|
\right).
\]
Similarly,
\[
|(\Hat{W}_{\BC}(\BC)-\Hat{W}_{\BC\BC}(\BI)(\BC-\BI),\BGx^{T}\BGx)|\le
C|\BC-\BI|^{2}|\BGx|^{2}
\]
When $\BF=\Grad\By(\Bx;h,\Gl)$ and $\BGx=\Grad\BGf$ we obtain, taking into
account (\ref{fundass}), that
\[
|\BF-\BI|\le C\Gl,\qquad|\BC-\BI|\le C\Gl.
\]
The estimate (\ref{basic}) now follows from the formulas
\[
4\Hat{W}_{\BC\BC}(\BI)=W_{\BF\BF}(\BI)=\SFL_{0},\qquad|\BC-\BI-2\Gl e(\Bu^{h})|\le C\Gl^{2}.
\]
\end{proof}
Let us show that for any $\epsilon>0$ there exists $h(\Ge)>0$, so that for all
$h<h(\Ge)$ there exists $\BGf_{h}\in\CA_{h}$, such that
$\Gd^{2}\CE(\BGf_h;h,\Hat{\Gl}(h)(1+2\Ge))<0$. In that case we will be able to
conclude that $\Gl(h)\le\Hat{\Gl}(h)(1+2\Ge)$ for any $h<h(\Ge)$.  For any
fixed $\Ge$ and any $h>0$ (for which $\CA_{h}$ is non-empty) there exists
$\BGf_h\in \CA_h$ such that
\begin{equation}
\label{1+epsilon}
\mathfrak{S}_{h}(\BGf_{h})\leq
\Hat{\Gl}(h)(1+\epsilon)\mathfrak{C}_{h}(\BGf_{h}),
\end{equation}
thus, putting $\lambda=\lambda_\epsilon(h)=\Hat{\Gl}(h)(1+2\epsilon)$ and
$\BGf=\BGf_h$ in (\ref{basic}) and utilizing (\ref{1+epsilon}) we get,
\[
\Gd^{2}\CE(\BGf_h;h,\Gl_\epsilon(h))\le-\frac{\epsilon}{(1+\epsilon)}
\mathfrak{S}_{h}(\BGf_{h})+C(\Hat{\Gl}(h)\|e(\BGf_h)\|\|\Grad\BGf_h\|
+\Hat{\Gl}(h)^{2}\|\Grad\BGf_h\|^{2}).
\]
By (\ref{Lcoerc}) we obtain
\begin{equation}
  \label{coerc}
\mathfrak{S}_{h}(\BGf_{h})\ge\Ga_{\SFL_{0}}\|e(\BGf_h)\|^{2}.
\end{equation}
Now, by the Korn inequality we obtain
\[
\Gd^{2}\CE(\BGf_h;h,\Gl_\epsilon(h))\leq\left(-\frac{\epsilon\alpha_{L_0}}{(1+\epsilon)}+
C\left(\frac{\Hat{\Gl}(h)}{\sqrt{K(V_{h})}}+\frac{\Hat{\Gl}(h)^{2}}{K(V_{h})}\right)
\right)\|e(\BGf_h)\|^2,
\]
We now see that due to (\ref{sufcond}) we have
$\Gd^{2}\CE(\BGf_h;h,\Gl_\epsilon(h))<0$ for sufficiently small $h$.

Now, let us show that for any $\Ge>0$ there exists $h(\Ge)>0$, so that for any
$h<h(\Ge)$, any $0<\lambda\leq\Hat{\Gl}(h)(1-\epsilon)$ and any $\BGf\in\CA_{h}$ we have
$\Gd^{2}\CE(\BGf;h,\Gl)>0$. Indeed, using (\ref{basic}) and the
generalized Korn inequality
\[
\mathfrak{S}_{h}(\BGf)\ge\Hat{\Gl}(h)\mathfrak{C}_{h}(\BGf)
\]
we estimate for $0<\lambda\leq\Hat{\Gl}(h)(1-\epsilon)$
\[
\Gd^{2}\CE(\BGf;h,\Gl)\geq\epsilon\mathfrak{S}_{h}(\BGf)-|C|(\Hat{\Gl}(h)\|e(\BGf)\|\|\Grad\BGf\|+\Hat{\Gl}(h)^{2}\|\Grad\BGf\|^{2})
\]
Using (\ref{coerc}) and the Korn inequality we conclude that
\[
\Gd^{2}\CE(\BGf;h,\Gl)\geq\left(\epsilon\alpha_{L_0}-C\left(\frac{\Hat{\Gl}(h)}{\sqrt{K(V_{h})}}
+\frac{\Hat{\Gl}(h)^{2}}{K(V_{h})}\right)\right)\|e(\BGf)\|^2.
\]
We now see that $\Gd^{2}\CE(\BGf;h,\Gl)>0$ for sufficiently small $h$, which
means that $\lambda(h)\geq \Hat{\Gl}(h)(1-\epsilon)$. Therefore, part (a) of
the theorem is proved.

We will now establish part (b). Assume now that $\BGf_h$ is a buckling mode. Then
$\Ga_{h}\to 0$, as $h\to 0$, where
\[
\alpha_h=\frac{\Gd^{2}\CE(\BGf_{h};h,\Gl(h))}{\Gl(h)
\mathfrak{C}_{h}(\BGf_{h})}.
\]
Observe that by virtue of $(a)$, the condition (\ref{sufcond}) holds for $\Hat{\Gl}(h)$ replaced by $\Gl(h),$ therefore by (\ref{sufcond}) and by the Korn inequality,
\begin{equation}
\label{0limit1}
\lims_{h\to 0}\frac{\lambda(h)\|e(\BGf_h)\|\|\Grad\BGf_h\|}{\mathfrak{S}_{h}(\BGf_{h})}\leq
\lims_{h\to 0}\frac{\lambda(h)\|e(\BGf_h)\|\|\Grad\BGf_h\|}{\alpha_{L_{0}}\|e(\BGf_h)\|^2}\leq
\lims_{h\to 0}\frac{\lambda(h)}{\alpha_{L_{0}}\sqrt{K(V_h)}}=0,
\end{equation}
and similarly
\begin{equation}
\label{0limit2}
\lim_{h\to 0}\frac{\lambda(h)^2\|\Grad\BGf_h\|^2}{\mathfrak{S}_{h}(\BGf_{h})}=0.
\end{equation}
Let us substitute $\BGf=\BGf_{h}$ and $\Gl=\Gl(h)$ into (\ref{basic}) and
divide by $\mathfrak{S}_{h}(\BGf_{h})$. Using
(\ref{0limit1}) and (\ref{0limit2}) we obtain
\[
\lim_{h\to 0}\left|1-\frac{\Gl(h)}{\mathfrak{K}(h,\BGf_{h})}
-\frac{\Gd^{2}\CE(\BGf_{h};h,\Gl(h))}{\mathfrak{S}_{h}(\BGf_{h})}\right|=0.
\]
We now use $\Ga_{h}$ to eliminate $\Gd^{2}\CE$:
\[
\frac{\Gd^{2}\CE(\BGf_{h};h,\Gl(h))}{\mathfrak{S}_{h}(\BGf_{h})}=
\frac{\Ga_{h}\Gl(h)}{\mathfrak{K}(h,\BGf_{h})}.
\]
Therefore,
\[
\lim_{h\to 0}\frac{(1+\Ga_{h})\Gl(h)}{\mathfrak{K}(h,\BGf_{h})}=1.
\]
Recalling that $\Ga_{h}\to 0$, as $h\to 0$ we conclude that
\[
\lim_{h\to 0}\frac{\Gl(h)}{\mathfrak{K}(h,\BGf_{h})}=1,
\]
and part (b) follows via part (a).

Let us prove part (c). Let $\BGf_{h}$ satisfy (\ref{bmode}). Then, by part
(a), $\Gb_{h}\to 0$, as $h\to 0$, where
\[
\beta_h=\frac{\mathfrak{K}(h,\BGf_h)}{\Gl(h)}-1.
\]
Therefore,
\begin{equation}
  \label{betah}
\Gl(h)\mathfrak{C}_{h}(\BGf_{h})=
\frac{1}{1+\Gb_{h}}\mathfrak{S}_{h}(\BGf_{h}).
\end{equation}
Let us substitute $\BGf=\BGf_{h}$ and $\Gl=\Gl(h)$ into (\ref{basic}) and
divide by $\Gl(h)\mathfrak{C}_{h}(\BGf_{h})$ to
obtain (using \ref{betah})
\begin{equation}
\label{0ineq}
\left|\frac{\Gd^{2}\CE(\BGf_h;h,\Gl(h))}{\Gl(h)\mathfrak{C}_{h}(\BGf_{h})}-\beta_h\right|\leq
\frac{C(1+\beta_h)(\lambda(h)\|e(\BGf_h)\|\|\Grad\BGf_h\|+
\lambda(h)^{2}\|\Grad\BGf_h\|^{2})}{\mathfrak{S}_{h}(\BGf_{h})}.
\end{equation}
Note that (\ref{0limit1}) and (\ref{0limit2}) continue to hold for any
$\BGf_{h}\in \CA_h,$. Thus (\ref{0ineq}) implies
\[
\lim_{h\to 0}\frac{\Gd^{2}\CE(\BGf_h;h,\Gl(h))}{\Gl(h)
\mathfrak{C}_{h}(\BGf_{h})}=0,
\]
therefore $\BGf_h$ is a buckling mode.
\end{proof}
The passage from the second variation of the trivial branch to the
constitutively linearized functional $\mathfrak{K}(h,\BGf)$ is an important
simplification. However, in specific problems we might want to simplify the
functional $\mathfrak{K}(h,\BGf)$ even further. For that reason we want to
establish general criteria for the validity of replacing one functional that
characterizes buckling by another. Behind this flexibility is the equivalence
relation on the space of functionals.
\begin{definition}
  \label{def:Bequivalence}
  Two pairs $(\CB_h, J_{1}(h,\BGf))$ and $(\CB_h, J_{2}(h,\BGf))$ are called
  \textbf{buckling equivalent} if the pair $(\CB_h,J_{1}(h,\BGf))$ characterizes buckling if
  and only if $(\CB_h,J_2(h,\BGf))$ does.
\end{definition}
The notion of B-equivalence was introduced in \cite{grtr07} on the set of
functions $m(h,\Gl)$ in order to capture buckling load by means of
constitutive linearization. Definition~\ref{def:Bequivalence} extends the idea
of buckling equivalence to \emph{functionals} in order to capture buckling
modes in addition to buckling loads.

Theorem~\ref{th:Bequivalence} below gives a convenient criterion of buckling equivalence.
\begin{theorem}[Buckling equivalence]
\label{th:Bequivalence}
If either
\begin{equation}
\label{J1J2}
\lim_{h\to 0}\Gl(h)\sup_{\BGf\in\CB_{h}}\left|\frac{1}{J_1(h,\BGf)}-\frac{1}{J_2(h,\BGf)}\right|=0,
\end{equation}
or
\begin{equation}
\label{J2J1}
\lim_{h\to 0}\nth{\Gl(h)}\sup_{\BGf\in\CB_{h}}|J_1(h,\BGf)-J_2(h,\BGf)|=0,
\end{equation}
then the pairs $(\CB_h, J_1(h,\BGf))$ and $(\CB_h, J_2(h,\BGf))$ are
buckling equivalent in the sense of Definition~\ref{def:Bequivalence}.
\end{theorem}
\begin{proof}
Let us introduce the notation
\[
\Hat{\Gl}_i(h)=\inf_{\BGf\in\CB_h}J_i(h,\BGf),\quad i=1,2.
\]
\[
\Gd_{1}(h)=\Gl(h)\sup_{\BGf\in\CB_{h}}\left|\frac{1}{J_1(h,\BGf)}-\frac{1}{J_2(h,\BGf)}\right|.
\]
\[
\Gd_{2}(h)=\nth{\Gl(h)}\sup_{\BGf\in\CB_{h}}|J_1(h,\BGf)-J_2(h,\BGf)|.
\]
Then
\[
\left|\frac{\Gl(h)}{\Hat{\Gl}_1(h)}-\frac{\Gl(h)}{\Hat{\Gl}_2(h)}\right|=
\Gl(h)\left|\sup_{\BGf\in\CB_{h}}\frac{1}{J_1(h,\BGf)}
-\sup_{\BGf\in\CB_{h}}\frac{1}{J_2(h,\BGf)}\right|\le\Gd_{1}(h)
\]
and
\[
\frac{|\Hat{\Gl}_1(h)-\Hat{\Gl}_{2}(h)|}{\Gl(h)}=
\nth{\Gl(h)}\left|\inf_{\BGf\in\CB_{h}}J_1(h,\BGf)
-\inf_{\BGf\in\CB_{h}}J_2(h,\BGf)\right|\le\Gd_{2}(h)
\]
Assume that $(\CB_h, J_1(h,\BGf))$ characterizes buckling. Then we have just proved
that if either $\Gd_{1}(h)\to 0$ or $\Gd_{2}(h)\to 0$, as $h\to 0$, then
$\Hat{\Gl}_2(h)/\Gl(h)\to 1$, as $h\to 0$, and condition (a) in the
Definition~\ref{def:equiv} is proved for $J_{2}(h,\BGf)$.

Observe that by parts (b) and (c) of Definition~\ref{def:equiv}
$\BGf_{h}\in\CB_{h}$ is the buckling mode \IFF
\[
\lim_{h\to 0}\frac{J_{1}(h,\BGf_{h})}{\Hat{\Gl}_{1}(h)}=1.
\]
This is equivalent to
\[
\lim_{h\to 0}\frac{\Gl(h)}{J_{1}(h,\BGf_{h})}=1.
\]
Therefore,
\[
\lim_{h\to 0}\frac{J_{2}(h,\BGf_{h})}{\Gl(h)}=1,
\]
since either
\[
\left|\frac{\Gl(h)}{J_{1}(h,\BGf_{h})}-\frac{\Gl(h)}{J_{2}(h,\BGf_{h})}\right|\le\Gd_{1}(h)
\]
or
\[
\frac{|J_{1}(h,\BGf_{h})-J_{2}(h,\BGf_{h})|}{\Gl(h)}\le\Gd_{2}(h)
\]
Thus, in view of part (a), $\BGf_{h}$ is a buckling mode \IFF
\[
\lim_{h\to 0}\frac{J_{2}(h,\BGf_{h})}{\Hat{\Gl}_{2}(h)}=1.
\]
\end{proof}

\section{Korn's inequality for the perfect cylindrical shell}
\setcounter{equation}{0}
\label{sec:KI}
Consider the perfect cylindrical shell given in cylindrical coordinates $(r,\Gth,z)$ as
\[
\CC_{h}=I_{h}\times\bb{T}\times[0,L],\qquad I_{h}=[1-h/2,1+h/2],
\]
where $\bb{T}$ is a 1-dimensional torus (circle) describing $2\pi$-periodicity in
$\Gth$. In this paper we consider the axial compression of the shell where
the displacement $\BGf\colon\CC_{h}\to \mathbb R^3$ satisfies one of the following two boundary conditions:
\begin{equation}
\label{BCfixed}
\phi_{z}(r,\Gth,0)=\phi_{r}(r,\Gth,0)=\phi_{\Gth}(r,\Gth,0)=
\phi_{r}(r,\Gth,L)=\phi_{\Gth}(r,\Gth,L)=0,
\end{equation}
or
\begin{equation}
\label{BCaverage}
\phi_{r}(r,\Gth,0)=\phi_{\Gth}(r,\Gth,0)=
\phi_{r}(r,\Gth,L)=\phi_{\Gth}(r,\Gth,L)=0,\quad \int_{I_{h}\times\bb{T}}\phi_{z}(r,\Gth,0)d\Gth dr=0.
\end{equation}
In the first case the top of the shell is allowed only the vertical
displacement and the bottom is kept fixed, while in the second case both the
top and the bottom of the shell are allowed the vertical displacements. In
order to eliminate vertical rigid body translations, the average vertical
displacement of the bottom edge is set to zero. In this paper we will work
almost exclusively with the \bc s (\ref{BCaverage}). In Section~\ref{sec:altbc} we will
show that our results can be extended to the \bc s
(\ref{BCfixed}). Accordingly, let
\begin{equation}
  \label{Vh}
V_{h}=\{\BGf\in W^{1,2}(\CC_{h};\mathbb R^3):\text{ (\ref{BCaverage}) holds}\}.
\end{equation}
\begin{equation}
  \label{Wh}
W_{h}=\{\BGf\in W^{1,2}(\CC_{h};\mathbb R^3):\text{ (\ref{BCfixed}) holds}\}\subset V_{h}.
\end{equation}
The theorem below establishes the asymptotics of the Korn constant $K(V_{h})$.
\begin{theorem}
\label{th:KI}
There exists a constant
$C(L)$ depending only on $L$ such that
\begin{equation}
  \label{KI}
  \|\Grad\Bu\|^{2}\le \frac{C(L)}{h\sqrt{h}}\|e(\Bu)\|^{2}
\end{equation}
For any $\Bu\in V_{h}$. Moreover, $K(V_{h})=C(L)h^{3/2}$.
  \end{theorem}
The theorem is proved in Appendix~\ref{app:KI}. We will also need the following
Korn-type inequalities whose proof is in Appendix~\ref{app:KI} as well.
\begin{lemma}
  \label{lem:KTI}
Suppose $u_{r}=u_{\Gth}=0$ at $z=0$ and $z=L$. Then
\begin{equation}
  \label{rtheta}
\|u_{z,\Gth}\|^{2}+\|u_{\Gth,z}\|^{2}\le 2\|e(\Bu)\|(\|e(\Bu)\|+\|u_{r}\|),
\end{equation}
and  there exist a
constant $C(L)>0$ depending only on $L$ and an absolute constant $h_0>0$
such that for all $h\in(0,h_0)$ and for all $L>0$
\begin{equation}
  \label{preK}
  \|\Grad\Bu\|^{2}\le C(L)\|e(\Bu)\|\left(\|e(\Bu)\|+\frac{\|u_{r}\|}{h}\right).
\end{equation}
\end{lemma}
We remark that the power of $h$ in the inequality (\ref{KI}) is
optimal. Indeed, let $n_{h}=[h^{-1/4}]$ (integer part of $h^{-1/4}$) and
$0<\eta_0<\pi$. Let $\Gvf(\eta,z)$ be a smooth compactly supported function on
 $(-\eta_{0},\eta_{0})\times(0,L)$. We define
 \begin{equation}
   \label{phiphi}
\phi^{h}(\Gth,z)=\Gvf(n_{h}\Gth,z),\quad\Gth\in[-\pi,\pi],\ z\in[0,L].   
 \end{equation}
Extended $2\pi$-periodically in $\Gth$ the function $\phi^{h}$ can be regarded
as a smooth function on $\bb{T}\times[0,L]$. We now define the ansatz
$\BU^{h}(r,\Gth,z)$ as follows
\begin{equation}
\begin{cases}
\label{ansatz0}
U^{h}_{r}(r,\Gth,z)=&-\phi^{h}_{,\Gth\Gth}(\Gth,z)\\[2ex]
U^{h}_{\Gth}(r,\Gth,z)=&r\phi^{h}_{,\Gth}(\Gth,z)+(r-1)\phi^{h}_{,\Gth\Gth\Gth}(\Gth,z),\\[2ex]
U^{h}_{z}(r,\Gth,z)=&(r-1)\phi_{,\Gth\Gth z}(\Gth,z)-\phi^{h}_{,z}(\Gth,z).
\end{cases}
\end{equation}
We compute 
\begin{equation}
  \label{ansgrad}
\lim_{h\to 0}h^{1/4}\|\Grad\BU^{h}\|^{2}=\|\Gvf_{,\eta\eta\eta}\|_{L^{2}(\bb{R}^{2})}^{2}.  
\end{equation}
while 
\begin{equation}
  \label{Eans}
\lim_{h\to 0}h^{-5/4}\|e(\BU^{h})\|^{2}=\|\Gvf_{,zz}\|_{L^{2}(\bb{R}^{2})}^{2}+
\nth{12}\|\Gvf_{,\eta\eta\eta\eta}\|_{L^{2}(\bb{R}^{2})}^{2},  
\end{equation}
producing
\[
K(V_{h})=O(h^{3/2}).
\]
Let $(\Grad\BU)_{\Ga\Gb}$ be the components of $\Grad\BU$ in cylindrical coordinates
\begin{equation}
  \label{gradcomp}
\Grad\BU=\sum_{\{\Ga,\Gb\}\subset\{r,\Gth,z\}}(\Grad\BU)_{\Ga\Gb}\Be_{\Ga}\otimes\Be_{\Gb}.  
\end{equation}
We compute
\begin{equation}
  \label{KCas}
  \|(\Grad\BU^{h})_{\Gth r}\|^{2}+\|(\Grad\BU^{h})_{r\Gth}\|^{2}=O\Big(\frac{1}{h^{1/4}}\Big)=
O\left(\frac{\|e(\BU^{h})\|^{2}}{K(V_{h})}\right),
\end{equation}
\begin{equation}
  \label{has}
  \|(\Grad\BU^{h})_{z r}\|^{2}+\|(\Grad\BU^{h})_{r z}\|^{2}=O(h^{1/4})=O\left(\frac{\|e(\BU^{h})\|^{2}}{h}\right)
\end{equation}
\begin{equation}
  \label{sqrthas}
  \|(\Grad\BU^{h})_{\Gth z}\|^{2}+\|(\Grad\BU^{h})_{z\Gth}\|^{2}=O(h^{3/4})=
O\left(\frac{\|e(\BU^{h})\|^{2}}{\sqrt{h}}\right),
\end{equation}
\begin{equation}
  \label{1as}
  \|(\Grad\BU^{h})_{\Gth\Gth}\|^{2}+\|(\Grad\BU^{h})_{zz}\|^{2}=O(h^{5/4})=O\left(\|e(\BU^{h})\|^{2}\right)
\end{equation}
In addition we also have
\begin{equation}
  \label{Uras}
\|U_{r}^{h}\|^{2}=O(h)=O\left(\frac{\|e(\BU^{h})\|^{2}}{h}\right).  
\end{equation}
\begin{remark}
\label{rem:anzatz0}
All functions $\phi(\eta,z)$ in (\ref{ansatz0}) constructed via (\ref{phiphi})
vanish together with all their derivatives at $z=0,L$, producing test functions
$\BU^{h}$ that satisfy the boundary conditions (\ref{BCfixed}) and thus also,
the boundary conditions (\ref{BCaverage}). Therefore, Theorem~\ref{th:KI} and
Lemma~\ref{lem:KTI} hold for the space $W_{h}$ defined by (\ref{Wh}).
\end{remark}

\section{Korn-like inequalities for gradient components}
\setcounter{equation}{0}
\label{sec:KLI}
In order to understand the buckling of the thin walled cylinders we also need
to estimate the $L^{2}$ norm of the individual components of $\Grad\Bu$
defined in (\ref{gradcomp}) in terms of $\|e(\Bu)\|^{2}$. In this section we
will prove that the asymptotics (\ref{KCas})--(\ref{Uras}) of gradient
components of the test function (\ref{ansatz0}) is optimal. In fact,
the inequalities
\[
\|(\Grad\Bu)_{\Gth\Gth}\|^{2}=\left\|\frac{u_{\Gth,\Gth}+u_{r}}{r}\right\|^{2}\le\|e(\Bu)\|^{2},
\qquad\|(\Grad\Bu)_{zz}\|^{2}=\|u_{z,z}\|^{2}\le\|e(\Bu)\|^{2}
\]
are obvious, while the inequalities
\[
\|(\Grad\Bu)_{r\Gth}\|^{2}=\left\|\frac{u_{r,\Gth}-u_{\Gth}}{r}\right\|^{2}\le
\frac{C(L)}{h\sqrt{h}}\|e(\Bu)\|^{2},\qquad
\|(\Grad\Bu)_{\Gth r}\|^{2}=\|u_{\Gth,r}\|^{2}\le\frac{C(L)}{h\sqrt{h}}\|e(\Bu)\|^{2}
\]
are the immediate consequence of the Korn inequality (\ref{KI}). The $L^{2}$
norms $\|(\Grad\Bu)_{rz}\|=\|u_{r,z}\|$ and $\|(\Grad\Bu)_{zr}\|$ are within
$\|e(\Bu)\|$ of each other, while the same is true for $\|(\Grad\Bu)_{\Gth
  z}\|=\|u_{\Gth,z}\|$ and $\|(\Grad\Bu)_{z\Gth}\|$.  Thus, in order to show
that the estimates (\ref{KCas})--(\ref{Uras}) are optimal it suffices to prove
the upper bounds on $\|u_{r,z}\|$, $\|u_{\Gth,z}\|$ and $\|u_{r}\|$.
\begin{lemma}
  \label{lem:KLI}
Suppose $u_{r}=u_{\Gth}=0$ at $z=0$ and $z=L$. Then there exists a constant
$C(L)$ depending only on $L$ such that
\begin{equation}
  \label{u}
  \|u_{r}\|^{2}\le\frac{C(L)}{h}\|e(\Bu)\|^{2}.
\end{equation}
\begin{equation}
  \label{rz}
  \|u_{r,z}\|^{2}\le\frac{C(L)}{h}\|e(\Bu)\|^{2},
\end{equation}
\begin{equation}
  \label{thetaz}
  \|u_{\Gth,z}\|^{2}\le\frac{C(L)}{\sqrt{h}}\|e(\Bu)\|^{2},
\end{equation}
\end{lemma}
\begin{proof}
  We first observe that the inequality (\ref{u}) follows from (\ref{rz}) via
  the Poincar\'e inequality, and that (\ref{thetaz}) is a direct consequence
  of (\ref{u}) and (\ref{rtheta}). Hence, we only need to prove (\ref{rz}).
  The proof is based on the Fourier series in $\Gth$ and $z$ variables.  For
  this purpose we need to extend $\Bu(r,\Gth,z)$ as a periodic function in
  $z\in\bb{R}$. Our boundary conditions suggest that $u_{r}$ and $u_{\Gth}$
  must be extended as odd $2L$-periodic functions, while $u_{z}$ will be
  extended as an even $2L$-periodic function. Such an extension results in
  $H^{1}(\bb{T}^{2}\times I_{h};\bb{R}^{3})$ functions whenever $\Bu\in
  H^{1}(\bb{T}\times[0,L]\times I_{h};\bb{R}^{3})$. Here $\bb{T}$ and
  $\bb{T}^{2}$ denote 1 and 2-dimensional tori, corresponding to
  $2\pi$-periodicity in $\Gth$ and $[0,2\pi]\times[-L,L]$-periodicity in
  $(\Gth,z)$, respectively. Denoting the periodic extension without relabeling
  we have
  \begin{equation}
    \label{Fourier}
\Bu(r,\Gth,z)=\sum_{m=0}^{\infty}\sum_{n\in\mathbb Z}\Bu^{(m,n)}(r,\Gth,z),
  \end{equation}
where
\[
\begin{cases}
  u_{r}^{(m,n)}=
\Hat{\phi}_{r}(r;m,n)\sin\left(\dfrac{\pi mz}{L}\right)e^{in\Gth},&
\Hat{\phi}_{r}(r;m,n)=\nth{\pi L}\int_{0}^{2\pi}\int_{0}^{L}
u_{r}\sin\left(\dfrac{\pi mz}{L}\right)e^{in\Gth}dzd\Gth\\[3ex]
u_{\Gth}^{(m,n)}=
\Hat{\phi}_{\theta}(r;m,n)\sin\left(\dfrac{\pi mz}{L}\right)e^{in\Gth},&
\Hat{\phi}_{\theta}(r;m,n)=\nth{\pi L}\int_{0}^{2\pi}\int_{0}^{L}
u_{\Gth}\sin\left(\dfrac{\pi mz}{L}\right)e^{in\Gth}dzd\Gth\\[3ex]
u_{z}^{(m,n)}=
\Hat{\phi}_{z}(r;m,n)\cos\left(\dfrac{\pi mz}{L}\right)e^{in\Gth},&
\Hat{\phi}_{z}(r;m,n)=\nth{\pi L}\int_{0}^{2\pi}\int_{0}^{L}
u_{z}\cos\left(\dfrac{\pi mz}{L}\right)e^{in\Gth}dzd\Gth.
\end{cases}
\]
We observe that our periodic extension has the property
\[
\Grad\Bu(r,\Gth,-z)=-\left[
  \begin{array}{rrr}
    1 & 0 & 0\\
    0 & 1 & 0\\
    0 & 0 & -1
  \end{array}
\right]\Grad\Bu(r,\Gth,z)\left[
  \begin{array}{rrr}
    1 & 0 & 0\\
    0 & 1 & 0\\
    0 & 0 & -1
  \end{array}
\right]
\]
in cylindrical coordinates, and hence
\[
\int_{0}^{2\pi}\int_{-L}^{L}|F_{ij}|^{2}d\Gth dz=2\int_{0}^{2\pi}\int_{0}^{L}|F_{ij}|^{2}d\Gth dz,
\]
for all cylindrical components $F_{ij}$ of $\Grad\Bu$.  Therefore, it is
sufficient to prove (\ref{rz}) for functions of the form
\[
\Bv^{(m,n)}(r,\Gth,z)=\left(f_{r}(r)\sin\left(\dfrac{\pi mz}{L}\right),
f_{\Gth}(r)\sin\left(\dfrac{\pi mz}{L}\right),f_{z}(r)\cos\left(\dfrac{\pi mz}{L}\right)
\right)e^{in\Gth}.
\]
Indeed,
\[
\|u_{r,z}\|^{2}=\pi L\sum_{m=1}^{\infty}\sum_{n\in\bb{Z}}\|u^{(m,n)}_{r,z}\|^{2}\le
\pi L\sum_{m=0}^{\infty}\sum_{n\in\bb{Z}}\frac{C(L)}{h}\|e(\Bu^{(m,n)})\|^{2}=
\frac{C(L)}{h}\|e(\Bu)\|^{2}.
\]

Observe that all functions of the form $\Bv^{(m,n)}$
satisfy the \bc s from Theorem~\ref{th:KI}. Therefore, Theorem~\ref{th:KI} and
Lemma~\ref{lem:KTI} are applicable to such functions. We now fix $m\ge 1$ and
$n\in\bb{Z}$ and for simplicity of notation we use $(v_{r},v_{\Gth},v_{z})$
instead of $(v_{r}^{(m,n)},v_{\Gth}^{(m,n)},v_{z}^{(m,n)})$.

We notice that if $\|v_{r}\|\le 3\|e(\Bv)\|$, then the inequality
(\ref{preK}) implies that
\[
\|v_{r,z}\|^{2}\le \|\Grad\Bv\|^{2}\le\frac{C(L)}{h}\|e(\Bv)\|^{2},
\]
and (\ref{rz}) is proved. Let us prove the inequality (\ref{rz}) under the
assumption that $\|v_{r}\|>3\|e(\Bv)\|$. In that case the inequalities
(\ref{rtheta}) and (\ref{preK}) become
\begin{equation}
  \label{rtheta1}
\|v_{z,\Gth}\|^{2}+\|v_{\Gth,z}\|^{2}\le \frac{8}{3}\|e(\Bv)\|\|v_{r}\|,
\end{equation}
\begin{equation}
  \label{preK1}
  \|\Grad\Bv\|^{2}\le \frac{C(L)}{h}\|e(\Bv)\|\|v_{r}\|.
\end{equation}
We estimate
\[
n^2\|v_r\|^2=\|v_{r,\theta}\|^2\leq 2\|v_{r,\theta}-v_\theta\|^2+2\|v_\theta\|^2
\leq 5\|\nabla \Bv\|^2+\frac{L^{2}}{\pi^{2}m^{2}}\|v_{\theta,z}\|^2\le
C(L)\|\nabla \Bv\|^2.
\]
Applying the inequality (\ref{preK1}) we obtain
\begin{equation}
\label{n^2urleqe(u)}
n^2\|v_r\|\leq \frac{C(L)}{h}\|e(\Bv)\|.
\end{equation}
We estimate
\[
\|v_r\|^2\leq 2\|v_r+v_{\theta,\theta}\|^2+2\|v_{\theta,\theta}\|^2
\leq 5\|e(\Bv)\|^2+2n^2\|v_\theta\|^2,
\]
and
\[
\frac{m^2\pi^2}{L^2}\|v_\theta\|^2=\|v_{\theta,z}\|^2\leq
\frac{8}{3}\|e(\Bv)\|\|v_r\|,
\]
due to (\ref{rtheta1}). Combining the two inequalities we obtain
\begin{equation}
  \label{vrineq1}
\|v_r\|^2\leq 5\|e(\Bv)\|^2+\frac{16L^{2}n^2}{3m^2\pi^2}\|e(\Bv)\|\|v_r\|.
\end{equation}
By our assumption $\|e(\Bv)\|^{2}<\|v_{r}\|^{2}/9$. We use this
inequality to estimate the first term on the \rhs\ of (\ref{vrineq1}) and obtain
\begin{equation}
\label{u_rleqe(u)2}
\|v_r\|\leq\frac{12L^{2}n^2}{m^2\pi^2}\|e(\Bv)\|.
\end{equation}
Using $\|e(\Bv)\|<\|v_{r}\|/3$ again, we conclude that $m\le C_{0}L|n|$ for some
absolute constant $C_{0}>0$. In particular, $n\not=0$, since $m\ge 1$. Finally,
multiplying now (\ref{n^2urleqe(u)}) and (\ref{u_rleqe(u)2}) we get
\[
m^2\|v_r\|^2\leq \frac{C(L)}{h}\|e(\Bv)\|^2,
\]
which completes the proof.
\end{proof}
Thus, we have established the following Korn-like inequalities for gradient
components.
\begin{theorem}
  \label{th:KLI}
Suppose $u_{r}=u_{\Gth}=0$ at $z=0$ and $z=L$. Then there exists a constant
$C(L)$ depending only on $L$ such that
  \begin{equation}
    \label{ththzz}
    \|(\Grad\Bu)_{\Gth\Gth}\|^{2}+\|(\Grad\Bu)_{zz}\|^{2}\le\|e(\Bu)\|^{2},
  \end{equation}
  \begin{equation}
    \label{rthr}
    \|(\Grad\Bu)_{r\Gth}\|^{2}+\|(\Grad\Bu)_{\Gth r}\|^{2}\le\frac{C(L)}{h\sqrt{h}}\|e(\Bu)\|^{2},
  \end{equation}
  \begin{equation}
    \label{urzr}
\|u_{r}\|^{2}+\|(\Grad\Bu)_{rz}\|^{2}+\|(\Grad\Bu)_{zr}\|^{2}\le\frac{C(L)}{h}\|e(\Bu)\|^{2},
  \end{equation}
  \begin{equation}
    \label{zthz}
\|(\Grad\Bu)_{\Gth z}\|^{2}+\|(\Grad\Bu)_{z\Gth}\|^{2}\le\frac{C(L)}{\sqrt{h}}\|e(\Bu)\|^{2},
  \end{equation}
\end{theorem}
\begin{remark}
  \label{rem:KLI}
Theorem~\ref{th:KLI} is obviously valid for $\Bu\in W_{h}\subset V_{h}$. The
inequalities (\ref{u})--(\ref{thetaz}) are also sharp in $W_{h}$, since, as we
mentioned in Remark~\ref{rem:anzatz0} the ansatz (\ref{ansatz0}), on which the
asymptotically behavior of gradient components is achieved contains many
functions in $W_{h}$.
\end{remark}

\section{Trivial branch in a perfect cylindrical shell}
\setcounter{equation}{0}
\label{sec:trbr}
By perfect cylinder we understand the set, given in cylindrical coordinates as
\[
\CC_{h}=\{(r,\Gth,z):r\in I_{h},\ \Gth\in\bb{T},\ z\in[0,L]\}.
\]
In order to describe the imposed \bc s and loading we need to specify the
space $V_{h}$ and the
functions $\Bt(\Bx;\Gl,h)$ and $\bra{\By}(\Bx;\Gl,h)$ in (\ref{energy}) and
(\ref{hdbc}), respectively. For the space $V_{h}$ given by (\ref{Vh}) we define
\begin{equation}
  \label{trac}
\Bt(\Bx;h,\Gl)=
\begin{cases}
  \Bzr,&r=1\pm\frac{h}{2},\ \Gth\in\bb{T},\ z\in(0,L),\\
\Gl\Be_{z},&r\in I_{h},\ \Gth\in\bb{T},\ z=0,\\
  -\Gl\Be_{z},&r\in I_{h},\ \Gth\in\bb{T},\ z=L.
\end{cases}
\end{equation}
We also have $\bra{\By}(\Bx;\Gl,h)=\Bx+\bra{\BU}(\Bx;\Gl,h)$, where
\begin{equation}
  \label{Ubar}
\bra{U}^{(r)}(r,\Gth,z;h,\Gl)=a(\Gl)r,\qquad \bra{U}^{(\Gth)}(r,\Gth,z;h,\Gl)=0,\qquad
\bra{U}^{(z)}(r,\Gth,z;h,\Gl)=0,
\end{equation}
where the explicit form of $a(\Gl)$ will be given below for a specific
energy satisfying properties (P1)--(P4).

We observe that during buckling the Cauchy-Green strain tensor
$\BC=\BF^{T}\BF$ is close to the
identity. Therefore, considering the energy which is quadratic in
$\BE=(\BC-\BI)/2$ should capture all the effects associated with
buckling. Hence, we assume, for the purposes of exhibiting the explicit form
of the trivial branch, that
\[
W(\BF)=\hf(\SFL_{0}\BE,\BE),\qquad\BE=\hf(\BF^{T}\BF-\BI).
\]
where the elastic tensor $\SFL_{0}$ is
isotropic. Following Koiter \cite{koit45} we consider the trivial branch
$\By(\Bx;h,\Gl)=\Bx+\BU(\Bx;h,\Gl)$ given in cylindrical coordinates by
\begin{equation}
  \label{trbr}
U^{(r)}=a(\Gl)r,\qquad U^{(\Gth)}=0,\qquad U^{(z)}=-b(\Gl)z,
\end{equation}
where the functions $a(\Gl)$ and $b(\Gl)$ will presently be determined. In
cylindrical coordinates we compute
\[
\Grad\BU=\left[
  \begin{array}{ccr}
    a & 0 &0\\
    0 & a& 0\\
    0 & 0& -b
  \end{array}
\right],\quad
\BF=\left[
  \begin{array}{ccc}
    1+a & 0 &0\\
    0 & 1+a& 0\\
    0 & 0& 1-b
  \end{array}
\right],\quad
\BE=\left[
  \begin{array}{ccc}
    a+\frac{a^{2}}{2} & 0 &0\\
    0 & a+\frac{a^{2}}{2}& 0\\
    0 & 0& \frac{b^{2}}{2}-b
  \end{array}
\right]
\]
Then we compute $\BP=\BF(\SFL_{0}\BE)$, and the traction-free condition
$\BP\Be_{r}=\Bzr$ on the lateral boundary is equivalent to the equation
\[
2a+a^{2}=\nu(2b-b^{2}),
\]
where $\nu$ is the Poisson's ratio for $\SFL_{0}$. The loading (\ref{trac})
implies that $(\BP\Be_{z},\Be_{z})=-\Gl$, which translates in the equation
\[
E(1-b)\frac{2\nu(2a+a^{2})+(1-\nu)(b^{2}-2b)}{2(1+\nu)(1-2\nu)}=-\Gl,
\]
where $E$ is the Young's modulus. Thus,
\[
a(\Gl)=\sqrt{1+\nu(2b(\Gl)-b(\Gl)^{2})}-1,
\]
where $b(\Gl)$ is the unique root of $Eb(1-b)(2-b)=2\Gl$, such that
$0<b(\Gl)<1-1/\sqrt{3}$. Such a root exists, whenever $0<\Gl<E/(3\sqrt{3})$.
We now see that the fundamental assumption (\ref{fundass}) is satisfied, since
the trivial branch parameters do not depend on $h$ explicitly. Choosing $\Gl$
as a loading parameter we obtain
\begin{equation}
  \label{perfectstr}
\BGs_{h}=\tns{\Be_{z}}=\left[
  \begin{array}{ccc}
    0 & 0 & 0\\
    0 & 0 & 0\\
    0 & 0 & 1
  \end{array}
\right].
\end{equation}
\begin{remark}
  \label{rem:trbr}
The same trivial branch (\ref{trac}), (\ref{Ubar}), (\ref{trbr}) also satisfies the more
restrictive fixed bottom \bc s (\ref{BCfixed}).
\end{remark}

\section{Buckling load and buckling mode for the perfect cylindrical shell}
\setcounter{equation}{0}
\label{sec:perfect}
Using the linearized stress (\ref{perfectstr}) in the Koiter trivial branch
(\ref{trbr}) we compute
\[
\int_{\CC_{h}}(\BGs_{h},\Grad\Bu^{T}\Grad\Bu)d\Bx=\|u_{r,z}\|^{2}+\|u_{z,z}\|^{2}+\|u_{\Gth,z}\|^{2}.
\]
Therefore, the space $\CA_{h}$ given by (\ref{Adef}) is simply the set of all
functions in $V_{h}$, given by (\ref{Vh}) that are not independent of
$z$-variable. On the one hand the estimates (\ref{rz}) and (\ref{thetaz})
imply that $\mathfrak{K}(h,\Bu)\ge c(L)h$, for any $\Bu\in\CA_{h}$ where
$\mathfrak{K}(h,\Bu)$ is given by (\ref{Kgen}). On the other, the test
function (\ref{ansatz0}) shows that
$\Hat{\Gl}(h)\le C(L)h$, where $\Hat{\Gl}(h)$ is given by (\ref{clin}). Thus,
\begin{equation}
  \label{h1scale}
c(L)h\le\Hat{\Gl}(h)\le C(L)h.
\end{equation}

In order to find the exact asymptotics of the buckling load as well as the
buckling mode we may simplify the functional $\mathfrak{K}(h,\Bu)$ by
observing that $\|u_{r,z}\|^{2}$ is much larger than $\|u_{z,z}\|^{2}$ and
$\|u_{\Gth,z}\|^{2}$, according to the estimates (\ref{rz}) and
(\ref{thetaz}), which are asymptotically saturated by (\ref{ansatz0}).
\begin{lemma}
  \label{lem:K1}
The pair $(\CA_h, \mathfrak{K}_{1}(h,\BGf))$ characterizes buckling, where
\[
\mathfrak{K}_{1}(h,\BGf)=\frac{\int_{\CC_{h}}(\SFL_{0}e(\BGf),e(\BGf))d\Bx}
{\int_{\CC_{h}}|\phi_{r,z}|^{2}d\Bx}.
\]
\end{lemma}
\begin{proof}
By (\ref{thetaz}), (\ref{coerc}) and (\ref{h1scale}) we have
\[
\left|\frac{1}{\mathfrak{K}(h,\BGf)}-\frac{1}{\mathfrak{K}_1(h,\BGf)}\right|=
\frac{\|\phi_{\Gth,z}\|^{2}+\|\phi_{z,z}\|^{2}}{\int_{\CC_{h}}(\SFL_{0}e(\BGf),e(\BGf))d\Bx}
\leq \frac{\frac{C(L)}{\sqrt{h}}+1}{\alpha_{L_0}}\leq\frac{c}{\sqrt{h}}=o\left(\nth{\Gl(h)}\right).
\]
Therefore by Theorem~\ref{th:Bequivalence} pair
$(\CA_h, \mathfrak{K}_1(h,\BGf))$ characterizes buckling.
\end{proof}
\begin{remark}
  \label{rem:K1}
Remarks~\ref{rem:anzatz0} and \ref{rem:KLI} imply that (\ref{h1scale}) and
hence Lemma~\ref{lem:K1} are valid for the fixed bottom \bc s (\ref{BCfixed}).
\end{remark}

\subsection{Bounds on the optimal wave numbers}

When $\SFL_{0}$ is isotropic the minimization of $\mathfrak{K}_1(h,\BGf)$ can
be done in Fourier space.
For any function $\Bf(r)=(f_{r}(r),f_{\Gth}(r),f_{\Gth}(r))$ and any $m\ge 0$
and $n\in\bb{Z}$ let
\[
\BGF_{m,n}(\Bf)=\left(f_{r}(r)\sin\left(\dfrac{\pi mz}{L}\right),
f_{\Gth}(r)\sin\left(\dfrac{\pi mz}{L}\right),f_{z}(r)\cos\left(\dfrac{\pi mz}{L}\right)
\right)e^{in\Gth}.
\]
For any $m\geq 0$ and $n\ge 0$ set
\begin{equation}
  \label{Xmndef}
 X(m,n)=
 \begin{cases}
\{\re(\BGF_{m,n}(\Bf)):\Bf\in C^{1}(I_{h};\bb{C}^{3})\},&n\ge 1\\
\{\BGF_{m,n}(\Bf):\Bf\in C^{1}(I_{h};\bb{R}^{3}),\int_{I_h}f_{z}(r)dr=0,\},&n=0.
 \end{cases}
\end{equation}
Observe that $X(m,n)\subset\CA_{h}$ for any integers $m\geq 1$ and
$n\ge 0$, since $\Bu\in X(m,n)$ is independent of $z$ \IFF $m=0$. Let
\begin{equation}
  \label{hatmn}
\Hat{\Gl}_{1}(h)=\inf_{\BGf\in\CA_{h}}\mathfrak{K}_1(h,\BGf),\qquad
\Hat{\Gl}(h;m,n)=\inf_{\BGf\in X(m,n)}\mathfrak{K}_1(h,\BGf).
\end{equation}
\begin{theorem}
  \label{th:mn}~
\begin{itemize}
\item[(i)] Let $\Hat{\Gl}_{1}(h)$ and $\Hat{\Gl}(h;m,n)$ be given by (\ref{hatmn}). Then
\begin{equation}
  \label{mninf}
\Hat{\Gl}_{1}(h)=\inf_{m\ge 1\atop n\ge 0}\Hat{\Gl}(h;m,n).
\end{equation}
The infimum in (\ref{mninf}) is attained at $m=m(h)$ and $n=n(h)$ satisfying
\begin{equation}
  \label{mnbound}
  m(h)\le\frac{C(L)}{\sqrt{h}},\qquad \frac{n(h)^{2}}{m(h)}\le\frac{C(L)}{\sqrt{h}}
\end{equation}
for some constant $C(L)$ depending only on $L$.
\item[(ii)] Suppose $m(h)$ and $n(h)$ are as in part (i). Then the pair
  $(X(m(h),n(h)),\mathfrak{K}_1(h,\BGf))$ characterizes buckling in the sense of Definition~\ref{def:equiv}.
\end{itemize}
\end{theorem}
\begin{proof}
Let us first prove (\ref{mninf}). Let
\[
\Ga(h)=\inf_{m\ge 1\atop n\ge 0}\Hat{\Gl}(h;m,n).
\]
It is clear that
$\Hat{\Gl}(h;m,n)\ge\Hat{\Gl}_{1}(h)$ for any $m\ge 1$ and $n\ge 0$, since
$X(m,n)\subset\CA_{h}$. Therefore, $\Ga(h)\ge\Hat{\Gl}_{1}(h)$. By definition
of $\Ga(h)$ we have
\begin{equation}
  \label{alphamn}
  \int_{\CC_{h}}(\SFL_{0}e(\BGf),e(\BGf))d\Bx\ge\Ga(h)\|\phi_{r,z}\|^{2}
\end{equation}
for any $\BGf\in X(m,n)$ and any $m\ge 1$ and $n\ge 0$.
Any $\BGf\in\CA_{h}$ can be expanded in the Fourier series in $\Gth$ and $z$
\[
\BGf(r,\Gth,z)=\sum_{m=0}^{\infty}\sum_{n=0}^{\infty}\BGf_{m,n}(r,\Gth,z),
\]
where $\BGf_{m,n}(r,\Gth,z)\in X(m,n)$. If $\SFL_{0}$ is isotropic elastic
tensor, or even more generally has the form
\[
(\SFL_{0}e,e)=Q_{1}(e_{rr},e_{r\Gth},e_{\Gth\Gth},e_{zz})+Q_{2}(e_{rz},e_{\Gth z}),
\]
where $Q_{1}(q_{1},q_{2},q_{3},q_{4})$ and $Q_{2}(q_{1},q_{2})$ are arbitrary
quadratic forms in their arguments, then the quadratic form $(\SFL_{0}e,e)$
diagonalizes in Fourier space, i.e.
\[
\int_{\CC_{h}}(\SFL_{0}e(\BGf),e(\BGf))d\Bx=\sum_{m=0}^{\infty}\sum_{n=0}^{\infty}
\int_{\CC_{h}}(\SFL_{0}e(\BGf_{m,n}),e(\BGf_{m,n}))d\Bx.
\]
We also have
\[
\|\phi_{r,z}\|^{2}=\sum_{m=0}^{\infty}\sum_{n=0}^{\infty}\|\phi_{r,z}^{(m,n)}\|^{2}.
\]
Therefore, (\ref{alphamn}) implies that
\[
\int_{\CC_{h}}(\SFL_{0}e(\BGf_{m,n}),e(\BGf_{m,n}))d\Bx\ge\Ga(h)\|\phi_{r,z}^{(m,n)}\|^{2}
\]
for every $m\ge 1$ and $n\ge 0$. Summing up, we obtain that
\[
\int_{\CC_{h}}(\SFL_{0}e(\BGf),e(\BGf))d\Bx\ge\Ga(h)\|\phi_{r,z}\|^{2}
\]
for every $\BGf\in\CA_{h}$. It follows that $\Hat{\Gl}_{1}(h)\ge\Ga(h)$ and
equality is proved.

Next we prove (\ref{mnbound}). We observe that, according to
Lemma~\ref{lem:K1}
\[
c(L)h\le\Hat{\Gl}_{1}(h)\le C(L)h.
\]
Then $\Hat{\Gl}(h;m,n)\ge\Hat{\Gl}_{1}(h)\ge c(L)h$ for any $m$ and $n.$ By
definition of the infimum, there exist indexes $m(h)$ and $n(h)$ such that
$\Hat{\Gl}(h;m(h),n(h))\le 2C(L)h$. By definition of the infimum there exists
$\BGf^{h}\in X(m(h),n(h))$ such that $\mathfrak{K}_1(h,\BGf^{h})\le 3C(L)h$.
Hence, there exists a possibly different constant $C(L)$ (not relabeled), such that
\begin{equation}
  \label{optrate}
\|e(\BGf^{h})\|^{2}\le C(L)h\|\phi^{h}_{r,z}\|^{2}=C(L)m(h)^{2}h\|\phi^{h}_{r}\|^{2}.
\end{equation}
To prove the first estimate in (\ref{mnbound}) we apply the inequality
(\ref{preK}) to $\BGf^{h}$ and estimate $\|e(\BGf^{h})\|$
via (\ref{optrate}).
\[
\frac{m(h)^{2}\pi^{2}}{L^{2}}\|\phi^{h}_{r}\|^{2}=\|\phi^{h}_{r,z}\|^{2}\le\|\Grad\BGf^{h}\|^{2}\le
C(L)\left(m(h)^{2}h+\frac{m(h)}{\sqrt{h}}\right)\|\phi^{h}_{r}\|^{2}.
\]
Hence
\[
h+\nth{m(h)\sqrt{h}}\ge c(L)
\]
for some constant $c(L)>0$, independent of $h$. Hence, the quantity
$m(h)\sqrt{h}$ must stay bounded, as $h\to 0$. To estimate $n(h)$ we write
\[
n(h)^{2}\|\phi^{h}_{r}\|^{2}=\|\phi^{h}_{r,\Gth}\|^{2}\le
C_{0}(\|(\Grad\BGf^{h})_{r\Gth}\|^{2}+\|\phi^{h}_{\Gth}\|^{2}).
\]
By the Poincar\'e inequality
\[
\|\phi^{h}_{\Gth}\|^{2}\le\frac{L^{2}}{\pi^{2}}\|\phi^{h}_{\Gth,z}\|^{2}\le
\frac{L^{2}}{\pi^{2}}\|(\Grad\BGf^{h})_{\Gth z}\|^{2},
\]
and hence
$
n(h)^{2}\|\phi^{h}_{r}\|^{2}\le C(L)\|(\Grad\BGf^{h})\|^{2}.
$
Applying (\ref{preK}) and estimating $\|e(\BGf^{h})\|$ via (\ref{optrate}) we
obtain
\[
n(h)^{2}\le C(L)\left(hm(h)^{2}+\frac{m(h)}{\sqrt{h}}\right),
\]
from which (\ref{mnbound})$_{2}$ follows via (\ref{mnbound})$_{1}$. The
boundedness of $m(h)$ and $n(h)$ implies that the minimum in (\ref{mninf}) is
attained. Part (i) is proved now.

To prove part (ii) it is sufficient to show, due to Lemma~\ref{lem:pairB_hJ},
that $X(m(h),n(h))$ contains a buckling mode. By definition of the infimum
in (\ref{hatmn}), for each $h\in(0,h_{0})$ there exists $\BGy_{h}\in
X(m(h),n(h))\subset\CA_{h}$ such that
\[
\Hat{\Gl}_{1}(h)=\Hat{\Gl}(h;m(h),n(h))\le\mathfrak{K}_1(h,\BGy_{h})\le\Hat{\Gl}_{1}(h)
+(\Hat{\Gl}_{1}(h))^{2}.
\]
Therefore,
\[
\lim_{h\to 0}\frac{\mathfrak{K}_1(h,\BGy_{h})}{\Hat{\Gl}_{1}(h)}=1.
\]
Hence, $\BGy_{h}\in X(m(h),n(h))$ is a buckling mode follows, since the pair
$(\CA_{h},\mathfrak{K}_1(h,\BGf))$ characterizes buckling.
\end{proof}

\subsection{Linearization in $r$}
In this section we prove that the buckling load and the buckling mode can be captured
by the test functions depending linearly on $r$. In fact we specify an
explicit structure that buckling modes should possess. We start by defining the
``linearization'' operator
\[
\CL(\Bu)=(v_{r}(\theta,z),ru_{\Gth}(1,\theta,z)-(r-1)v_{r,\Gth}(\theta,z),u_{z}(1,\theta,z)-(r-1)v_{r,z}(\theta,z)),
\]
where
\[
v_{r}(\theta,z)=\Av{I_{h}}u_{r}(r,\theta,z)dr.
\]
Define the space of vector fields
\begin{multline}
  \label{Xlin}
  X_{\rm lin}=\{(f(\Gth,z),rg(\Gth,z)-(r-1)f_{,\Gth}(\Gth,z),h(\Gth,z)-(r-1)f_{,z}(\Gth,z)):\\
(f,g,h)\in H^{1}(\bb{T}\times[0,L];\bb{R}^{3}),\ f(\Gth,0)=g(\Gth,0)=f(\Gth,L)=g(\Gth,L)=0, \int_{0}^{2\pi}h(\Gth,0)d\Gth=0\}.
\end{multline}
Incidentally, the test functions (\ref{ansatz0}) belong to $X_{\rm lin}$. It is also clear
that  $ X_{\rm lin}\subset V_h.$ We
also observe that if $\Bu\in X(m,n)$, then $\CL(\Bu)\in X(m,n)$. Let us show
that if $\BGy_{h}\in X(m(h),n(h))$ is a buckling mode then so is
$\CL(\BGy_{h})$.
  \begin{theorem}
  \label{th:lin}
Suppose $m(h)\ge 1$ and $n(h)\ge 0$ are as in part (i) of
Theorem~\ref{th:mn}. Let $\BGy_{h}\in X(m(h),n(h))$ be a buckling mode. Then
\begin{equation}
  \label{linu}
\mathfrak{K}_{1}(h,\CL(\BGy_{h}))\le\mathfrak{K}_{1}(h,\BGy_{h})(1+C(L)\sqrt h).
\end{equation}
\end{theorem}
\begin{proof}
We will perform linearization in $r$ sequentially. First in $u_{r}$, then in
$u_{\Gth}$ and finally in $u_{z}$. For this purpose we introduce the following
operators of ``partial linearization''
\[
\Bu_{1}=\CL_{r}(\Bu)=(v_{r}(\theta,z),u_{\Gth}(r,\theta,z),u_{z}(r,\theta,z)),
\]
\[
\Bu_{2}=\CL_{r,\Gth}(\Bu)=(v_{r}(\theta,z),ru_{\Gth}(1,\theta,z)-(r-1)v_{r,\Gth}(\theta,z),u_{z}(r,\theta,z)),
\]
where
\[
v_{r}(\theta,z)=\Av{I_{h}}u_{r}(r,\theta,z)dr.
\]
For any $\Bu\in\CA_{h}$ we have
\begin{equation}
\label{e(A1)e(A)}
\|e(\Bu_{1})-e(\Bu)\|\le 2(\|v_{r,\theta}-u_{r,\theta}\|+\|v_{r,z}-u_{r,z}\|+\|v_{r}-u_{r}\|).
\end{equation}
By the Poincar\'e inequality we have
\begin{equation}
\label{urlinur}
\|v_{r}-u_{r}\|^2\leq C_{0}h^2\|u_{r,r}\|^2\leq C_{0}h^2 \|e(\Bu)\|^2.
\end{equation}
Now let $\Bu\in X(m(h),n(h))$, where $m(h)$ and $n(h)$ satisfy (\ref{mnbound}). Then,
\begin{equation}
  \label{uvest}
\|v_{r,\theta}-u_{r,\theta}\|=n(h)\|v_{r}-u_{r}\|,\qquad
\|v_{r,z}-u_{r,z}\|=\frac{m(h)\pi}{L}\|v_{r}-u_{r}\|.
\end{equation}
Substituting this and (\ref{urlinur}) into (\ref{e(A1)e(A)}), we get
\[
\|e(\Bu_{1})-e(\Bu)\|\le C_{0}h\left(1+n(h)+\frac{m(h)\pi}{L}\right)\|e(\Bu)\|.
\]
Taking into account (\ref{mnbound}) we obtain
\[
\|e(\Bu_{1})-e(\Bu)\|\le C(L)\sqrt{h}\|e(\Bu)\|.
\]
Therefore
\begin{equation}
\label{e(u1)e(u)}
\int_{\CC_{h}}(\SFL_{0}e(\Bu_1),e(\Bu_1))d\Bx\leq(1+C(L)\sqrt{h})\int_{\CC_{h}}(\SFL_{0}e(\Bu),e(\Bu))d\Bx.
\end{equation}

We now make the next step in the linearization in $r$ and consider
$\Bu_2=\CL_{r,\Gth}(\Bu)$. Observe that $e(\Bu_{2})_{r\Gth}=0$. We also see
that
\begin{equation}
  \label{rrz}
e(\Bu_{2})_{rr}=e(\Bu_{1})_{rr},\quad e(\Bu_{2})_{zr}=e(\Bu_{1})_{zr},\quad
e(\Bu_{2})_{zz}=e(\Bu_{1})_{zz}.
\end{equation}
The remaining components are estimated as follows
\begin{equation}
  \label{ththz}
\|e(\Bu_{2})_{\Gth\Gth}-e(\Bu_{1})_{\Gth\Gth}\|\le2\|u_{\theta,\theta}^{(2)}-u_{\theta,\theta}\|,\qquad
\|e(\Bu_{2})_{\Gth z}-e(\Bu_{1})_{\Gth z}\|\le\|u_{\theta,z}^{(2)}-u_{\theta,z}\|.
\end{equation}
Therefore,
 \begin{equation}
 \label{e(A2)e(A)}
 \|e(\Bu_2)\|^{2}\le \|e(\Bu_{1})\|^{2}+
C_{0}(\|u_{\theta,\theta}^{(2)}-u_{\theta,\theta}\|^{2}+\|u_{\theta,z}^{(2)}-u_{\theta,z}\|^{2}).
 \end{equation}
We can estimate
\begin{equation}
  \label{step2th}
\|u_{\Gth,\Gth}^{(2)}-u_{\Gth,\Gth}\|^{2}=n(h)^{2}\|u_{\Gth}^{(2)}-u_{\Gth}\|^{2}\le
\frac{C(L)}{h}\|u_{\Gth}^{(2)}-u_{\Gth}\|^{2},
\end{equation}
due to (\ref{mnbound}). Similarly,
\begin{equation}
  \label{step2z}
\|u_{\Gth,z}^{(2)}-u_{\Gth,z}\|^{2}=\frac{\pi^{2} m(h)^{2}}{L^{2}}\|u_{\Gth}^{(2)}-u_{\Gth}\|^{2}\le
\frac{C(L)}{h}\|u_{\Gth}^{(2)}-u_{\Gth}\|^{2},
\end{equation}
We now proceed to estimate $\|u_{\Gth}^{(2)}-u_{\Gth}\|$. Let
\[
w(r,\Gth,z)=u_{\Gth,r}+v_{r,\Gth}-u_{\Gth}=2e(\Bu_{1})_{r\Gth}+\frac{1-r}{r}(v_{r,\Gth}-u_{\Gth}).
\]
Therefore,
\[
\|w\|^{2}\le8\|e(\Bu_{1})\|^{2}+h^{2}\|\frac{v_{r,\Gth}-u_{\Gth}}{r}\|^{2}\le
8\|e(\Bu_{1})\|^{2}+C(L)\sqrt{h}\|e(\Bu_{1})\|^{2}.
\]
due to the Korn inequality (\ref{KI}). Thus, $\|w\|\le C(L)\|e(\Bu_{1})\|$.
We can express $u_{\Gth}^{(2)}-u_{\Gth}$ in terms of $w$ as follows
\[
u_{\Gth}-u_{\Gth}^{(2)}=\int_1^rw(t,\Gth,z)dt+\int_1^r(u_\Gth(t,\Gth,z)-u_\Gth(1,\Gth,z))dt.
\]
Using the Cauchy-Schwartz inequality we have
\[
\int_{I_{h}}\left(\int_{1}^{r}f(t)dt\right)^{2}dr\le\frac{h^{2}}{4}\int_{I_{h}}f(r)^{2}dr.
\]
Therefore,
\[
\|u_{\Gth}-u_{\Gth}^{(2)}\|^{2}\le \frac{h^{2}}{2}(\|w\|^{2}+\|u_{\Gth}-u_\Gth(1,\Gth,z)\|^{2}).
\]
By the Poincar\'e inequality followed by the application of the Korn
inequality (\ref{KI}) we obtain
\[
\|u_{\Gth}-u_\Gth(1,\Gth,z)\|^{2}\le h^{2}\|u_{\Gth,r}\|^{2}\le C(L)\sqrt{h}\|e(\Bu_{1})\|^{2}.
\]
Therefore, we conclude that
\[
\|u_{\Gth}-u_{\Gth}^{(2)}\|^{2}\le C(L)h^{2}\|e(\Bu_{1})\|^{2}.
\]
Hence, (\ref{e(A2)e(A)}) becomes
\begin{equation}
  \label{u1u2}
\|e(\Bu_2)\|^{2}\le \|e(\Bu_{1})\|^{2}(1+C(L)h).
\end{equation}
Recalling (\ref{rrz}) and (\ref{ththz}) we get
\[
\|\Trc(e(\Bu_{2}))-\Trc(e(\Bu_{1}))\|\le C(L)\sqrt{h}\|e(\Bu_{1})\|.
\]
Therefore,
\begin{equation}
  \label{Tru1u2}
\|\Trc(e(\Bu_{2}))\|^{2}\le\|\Trc(e(\Bu_{1}))\|^{2}+C(L)h\|e(\Bu_{1})\|^{2}.
\end{equation}
When $\SFL_{0}$ is isotropic we have
\begin{equation}
\label{L0(e),e}
\int_{\CC_{h}}(\SFL_{0}e(\Bu_{2}),e(\Bu_{2}))d\Bx=\lambda\|\Trc(e(\Bu_{2}))\|^2+2\mu\|e(\Bu_{2})\|^2,
\end{equation}
where $\lambda$ and $\mu$ are the Lam\'e constants. The inequalities
(\ref{u1u2}) and (\ref{Tru1u2}) imply, using the coercivity of $\SFL_{0}$,
\begin{equation}
  \label{step2}
\int_{\CC_{h}}(\SFL_{0}e(\Bu_{2}),e(\Bu_{2}))d\Bx\le(1+C(L)h)
\int_{\CC_{h}}(\SFL_{0}e(\Bu_{1}),e(\Bu_{1}))d\Bx.
\end{equation}

In the last step of linearization we let $\Bv=\CL(\Bu)$. We compute $e(\Bv)_{rz}=0$ and
\[
  e(\Bv)_{rr}=e(\Bu_{2})_{rr},\quad e(\Bv)_{r\Gth}=e(\Bu_{2})_{r\Gth},\quad
e(\Bv)_{\Gth\Gth}=e(\Bu_{2})_{\Gth\Gth}.
\]
We also have
\begin{equation}
  \label{zztherr}
  \|e(\Bv)_{\Gth z}-e(\Bu_{2})_{\Gth z}\|\le 2\|v_{z,\Gth}-u_{z,\Gth}\|,\qquad
\|e(\Bv)_{zz}-e(\Bu_{2})_{zz}\|\le \|v_{z,z}-u_{z,z}\|.
\end{equation}
Analogously to (\ref{step2th}) and (\ref{step2z}) we have
\begin{equation}
  \label{step3thz}
  \|e(\Bv)_{\Gth z}-e(\Bu_{2})_{\Gth z}\|^{2}\le\frac{C(L)}{h}\|v_{z}-u_{z}\|^{2},\qquad
\|e(\Bv)_{zz}-e(\Bu_{2})_{zz}\|^{2}\le\frac{C(L)}{h}\|v_{z}-u_{z}\|^{2}.
\end{equation}

Integrating the equality $u_{z,r}=2e(\Bu_{2})_{rz}-v_{r,z}$ from 1 to $r$ we get
\[
u_z(r,\Gth,z)-v_{z}=2\int_0^r e(\Bu_{2})_{rz}(t,\Gth,z)dt.
\]
Thus, $\|v_{z}-u_z\|^2\leq h^2\|e(\Bu_{2})\|^2.$ Applying this estimate to
(\ref{step3thz}) we obtain
\[
\|e(\Bv)_{\Gth z}-e(\Bu_{2})_{\Gth z}\|^{2}\le C(L)h\|e(\Bu_{2})\|^{2},\qquad
\|e(\Bv)_{zz}-e(\Bu_{2})_{zz}\|^{2}\le C(L)h\|e(\Bu_{2})\|^{2}.
\]
We conclude that
\[
\|e(\Bv)\|^{2}\le \|e(\Bu_{2})\|^{2}(1+C(L)h),\qquad
\|\Trc(e(\Bv))\|^{2}\le\|\Trc(e(\Bu_{2}))\|^{2}+C(L)h\|e(\Bu_{2})\|^{2},
\]
and hence for the isotropic and coercive elastic tensor $\SFL_{0}$ we have
\begin{equation}
  \label{step3}
\int_{\CC_{h}}(\SFL_{0}e(\Bv),e(\Bv))d\Bx\leq
(1+C(L) h)\int_{\CC_{h}}(\SFL_{0}e(\Bu_{2}),e(\Bu_{2}))d\Bx.
\end{equation}

Finally to prove (\ref{linu}) we need to relate $\|u_{r,z}\|$ and
$\|v_{r,z}\|$.
We estimate $\|v_{r,z}\|\geq\|u_{r,z}\|-\|v_{r,z}-u_{r,z}\|$. Applying
(\ref{mnbound}) and (\ref{urlinur}) to (\ref{uvest}) we obtain
\[
\|v_{r,z}-u_{r,z}\|\le C(L)\sqrt h\|e(\Bu)\|,
\]
and hence, by (\ref{urlinur}) and (\ref{mnbound}),
\[
\|v_{r,z}\|\geq\|u_{r,z}\|-C(L)\sqrt h\|e(\Bu)\|.
\]
At this point the assumption that $\Bu\in X(m(h),n(h))$, where $m(h)$
and $n(h)$ satisfy (\ref{mnbound}) is insufficient. We also have to assume that
$\Bu=\BGy_{h}\in X(m(h),n(h))$ is a buckling mode. Recalling that the pair
$(X(m(h),n(h)),\mathfrak{K}_{1}(h,\BGf))$ characterizes buckling, we obtain
the inequality
$$\|e(\Bu)\|^2\leq C\mathfrak{K}_{1}(h,\Bu)\|u_{r,z}\|^2\leq
C\Hat{\Gl}(h)\|u_{r,z}\|^2\le Ch\|u_{r,z}\|^2.$$
Thus,
\begin{equation}
  \label{vrzurz}
\|v_{r,z}\|\geq \|u_{r,z}\|(1-C(L)h).
\end{equation}
Combining (\ref{e(u1)e(u)}), (\ref{step2}),  (\ref{step3}) and (\ref{vrzurz}) we obtain
(\ref{linu}).
\end{proof}
Introducing the notation $X_{\rm lin}(m,n)=X_{\rm lin}\cap X(m,n)$ we have the
following corollary of Theorem~\ref{th:lin}.
\begin{corollary}
\label{cor:main}
Let the integers  $m(h)\geq 1$ and  $n(h)$ be as in part (i) of Theorem~\ref{th:mn}.
Then the pair $(X_{\rm lin}(m(h), n(h)),\mathfrak{K}_{1}(h,\BGf))$ characterizes
buckling.
\end{corollary}
\begin{proof}
By Lemma~\ref{lem:pairB_hJ} it is sufficient to show that $X_{\rm lin}(m(h),
n(h))$ contains a buckling mode. Let $\BGy_{h}\in X(m(h),n(h))$ be a buckling
mode. Let us show that $\CL(\BGy_{h})\in X_{\rm lin}(m(h), n(h))$ is also a
buckling mode. Indeed, by Theorem~\ref{th:lin}
\[
1\le\frac{\mathfrak{K}_{1}(h,\CL(\BGy_{h}))}{\Hat{\Gl}_{1}(h)}\le
(1+C(L)\sqrt{h})\frac{\mathfrak{K}_{1}(h,\BGy_{h})}{\Hat{\Gl}_{1}(h)}.
\]
Taking a limit as $h\to 0$ and using the fact that $\BGy_{h}$ is a buckling
mode, we obtain
\[
\lim_{h\to 0}\frac{\mathfrak{K}_{1}(h,\CL(\BGy_{h}))}{\Hat{\Gl}_{1}(h)}=1.
\]
Hence, $\CL(\BGy_{h})$ is also a buckling mode, since the pair $(X(m(h),
n(h)),\mathfrak{K}_{1}(h,\BGf))$ characterizes buckling.
\end{proof}

\subsection{Algebraic simplification}
\label{sub:alg}
At this point the problem of finding the buckling load and a buckling mode can
be stated as follows. We first compute
\begin{equation}
  \label{almost}
  \Gl_{\rm lin}(h;m,n)=\inf_{\Bu\in X_{\rm lin}(m,n)}\mathfrak{K}_{1}(h,\Bu).
\end{equation}
The we find $m(h)$ and $n(h)$ as minimizers in
\begin{equation}
  \label{ldlin}
  \Gl_{\rm lin}(h)=\min_{m\ge 1\atop n\ge 0}\Gl_{\rm lin}(h;m,n).
\end{equation}
For any $\Bu\in X_{\rm lin}(m,n)$ the integral in $r$ over $I_{h}$ can be computed
explicitly and the minimization problem (\ref{almost}) can be reduced to an
algebraic problem via the Fourier expansion (\ref{Fourier}). However, the
explicit expressions one obtains are, to use an understatement,
unwieldy. Therefore, for the purposes of simplifying the algebra, we replace
$e(\Bu)$ in the numerator of the functional $\mathfrak{K}_{1}(h,\Bu)$ by
\[
E(\Bu)=\hf(G(\Bu)+G(\Bu)^{T}),\qquad G(\Bu)=\left[
  \begin{array}{ccc}
    u_{r,r} & \dfrac{u_{r,\Gth}-u_{\Gth}}{r} & u_{r,z}\\[2ex]
    u_{\Gth,r} & u_{\Gth,\Gth}+u_{r} & u_{\Gth,z}\\[2ex]
    u_{z,r} & u_{z,\Gth} & u_{z,z}
  \end{array}
\right].
\]
Let
\[
\mathfrak{K}_{0}(h,\Bu)=\frac{\int_{\CC_{h}}r^{-1}(\SFL_{0}E(\Bu),E(\Bu))d\Bx}{\|u_{r,z}\|^{2}}.
\]
\begin{theorem}
  \label{th:penultsimp}
The pair $(\CA_h, \mathfrak{K}_{0}(h,\Bu))$ characterizes buckling.
\end{theorem}
\begin{proof}
First we observe that
\begin{equation}
  \label{Jacob}
 \left|\int_{\CC_{h}}\left(\nth{r}-1\right)(\SFL_{0}E(\Bu),E(\Bu))d\Bx\right|\le
h\int_{\CC_{h}}(\SFL_{0}E(\Bu),E(\Bu))d\Bx.
\end{equation}
We easily see that $|E(\Bu)-e(\Bu)|\le h|\Grad\Bu|$.
Therefore, by the Korn inequality (\ref{KI}) we obtain
\[
\|E(\Bu)-e(\Bu)\|\le C(L)h^{1/4}\|e(\Bu)\|.
\]
This inequality also implies that
\[
\|e(\Bu)\|\le C(L)h^{1/4}\|e(\Bu)\|+\|E(\Bu)\|,
\]
from which we conclude that $\|e(\Bu)\|\le C(L)E(\Bu)$ for sufficiently small $h$.
Therefore,
\begin{equation}
  \label{Evse}
\|E(\Bu)-e(\Bu)\|\le C(L)h^{1/4}\|E(\Bu)\|,
\end{equation}
and
\[
\left|\int_{\CC_{h}}[(\SFL_{0}e(\Bu),e(\Bu))-(\SFL_{0}E(\Bu),E(\Bu))]d\Bx\right|\le
C(L)\|E(\Bu)-e(\Bu)\|\|E(\Bu)\|\le C(L)h^{1/4}\|E(\Bu)\|^{2}
\]
Hence,
\[
\left|\int_{\CC_{h}}(\SFL_{0}e(\Bu),e(\Bu))d\Bx-\int_{\CC_{h}}r^{-1}(\SFL_{0}E(\Bu),E(\Bu))d\Bx\right|
\le C(L)h^{1/4}\int_{\CC_{h}}r^{-1}(\SFL_{0}E(\Bu),E(\Bu))d\Bx.
\]
Therefore,
\begin{equation}
\label{ineqk0ki}
\left|\nth{\mathfrak{K}_{0}(h,\Bu)}-\nth{\mathfrak{K}_{1}(h,\Bu)}\right|\le
\frac{C(L)h^{1/4}}{\mathfrak{K}_{1}(h,\Bu)}.
\end{equation}
It follows that
\[
\Gl(h)\sup_{\Bu\in\CA_{h}}\left|\nth{\mathfrak{K}_{0}(h,\Bu)}-\nth{\mathfrak{K}_{1}(h,\Bu)}\right|\le C(L)h^{1/4}\frac{\Gl(h)}{\Hat{\Gl}_{1}(h)}.
\]
We conclude that condition (\ref{J1J2}) is satisfied, since
$(\CA_h, \mathfrak{K}_{1}(h,\Bu))$ characterizes buckling. Then, by
Theorem~\ref{th:Bequivalence} the pair $(\CA_h, \mathfrak{K}_{0}(h,\Bu))$
characterizes buckling.
  \end{proof}
\begin{remark}
  \label{rem:K0}
The proof of Theorem~\ref{th:penultsimp} uses only the Korn inequality
(\ref{KI}). Therefore, it is also valid for the fixed bottom \bc s (\ref{BCfixed}).
\end{remark}
Recalling that $X_{\rm lin}(m(h),n(h))$ contains a buckling mode, we have the following
corollary of Theorem~\ref{th:penultsimp}.
\begin{corollary}
  The pair $(X_{\rm lin}(m(h),n(h)),\mathfrak{K}_{0}(h,\BGf))$ characterizes buckling.
\end{corollary}
The linearization and passage to the Fourier space make it convenient to
introduce the following notation.
\[
\CC_{0}=\{\Bx(r,\Gth,z):r=1,\ \Gth\in\bb{T},\ z\in[0,L]\}
\]
is the mid-surface of the undeformed cylindrical shell. For $\Bf\in
W^{1,2}(\CC_{0};\bb{R}^{3})$ we define
\begin{equation}
  \label{ulin}
\Bu=\CU(\Bf),\qquad
\begin{cases}
u_{r}=f_{r}(\Gth,z),\\
u_{\Gth}=rf_{\Gth}(\Gth,z)-(r-1)f_{r,\Gth}(\Gth,z),\\
u_{z}=f_{z}(\Gth,z)-(r-1)f_{r,z}(\Gth,z).
\end{cases}
\end{equation}
For $m\ge 1$, $n\ge 0$ and $\Hat{\Bf}\in\bb{C}^{3}$ we define
\begin{equation}
  \label{Fmode}
\Bf(\Gth,z)=\CF_{m,n}(\Hat{\Bf}),\qquad
\begin{cases}
f_{r}(\Gth,z)=\re\left(\Hat{f}_{r}\sin\left(\frac{\pi mz}{L}\right)e^{in\Gth}\right),\\
f_{\Gth}=\re\left(\Hat{f}_{\Gth}\sin\left(\frac{\pi mz}{L}\right)e^{in\Gth}\right),\\
f_{z}=\re\left(\Hat{f}_{z}\cos\left(\frac{\pi mz}{L}\right)e^{in\Gth}\right).
\end{cases}
\end{equation}
We also define
\[
\CU_{m,n}(\Hat{\Bf})=\CU(\CF_{m,n}(\Hat{\Bf})),\qquad\Hat{\Bf}\in\bb{C}^{3}.
\]

We compute
\[
\mathfrak{K}_{0}(h,\CU(\Bf))=\mu\frac{Q_{0}(\Bf)+\frac{h^{2}}{12}Q_{1}(\Bf)}{B(\Bf)},\qquad
B(\Bf)=\int_{\CC_{0}}|f_{r,z}|^{2}dzd\Gth,
\]
where $\mu$ is the shear modulus and
\[
Q_{0}(\Bf)=\int_{\CC_{0}}\{\GL|f_{\Gth,\Gth}+f_{z,z}+f_{r}|^{2}+2|f_{\Gth,\Gth}+f_{r}|^{2}
+2|f_{z,z}|^{2}+|f_{z,\Gth}+f_{\Gth,z}|^{2}\}d\Bx,
\]
\[
Q_{1}(\Bf)=\int_{\CC_{0}}\{\GL|f_{r,zz}+f_{r,\Gth\Gth}-f_{\Gth,\Gth}|^{2}
+2|f_{r,\Gth\Gth}-f_{\Gth,\Gth}|^{2}+2|f_{r,zz}|^{2}+|f_{\Gth,z}-2f_{r,\Gth z}|^{2}\}d\Bx,
\]
where $\GL=2\nu/(1-2\nu)$ and $\nu$ is the
Poisson ratio. The problem of finding the buckling load and buckling load is stated as
(\ref{almost})--(\ref{ldlin}), where the functional $\mathfrak{K}_{1}(h,\BGf)$
is replaced with $\mathfrak{K}_{0}(h,\BGf)$.

When $\Bu=\CU_{m,n}(\Hat{\Bf})\in X_{\rm lin}(m,n)$ the problem (\ref{almost})
is purely algebraic and can be solved explicitly. However, the minimization in
(\ref{ldlin}) is a bit messy. In fact, the functional
$\mathfrak{K}_{0}(h,\BGf)$ can be simplified further, yielding a very simple
algebraic problem for computing the buckling load and the buckling mode. For
$\BGf=\CU(\Bf)$ we define
\[
\mathfrak{K}^{*}(h,\BGf)=\mu\frac{Q_{0}(\Bf)+\frac{h^{2}}{12}Q_{1}^{*}(\Bf)}{B(\Bf)},\qquad
B(\Bf)=\int_{\CC_{0}}|f_{r,z}|^{2}dzd\Gth,
\]
\[
Q_{1}^{*}(\Bf)=\int_{\CC_{0}}\{\GL|f_{r,zz}+f_{r,\Gth\Gth}|^{2}
+2|f_{r,\Gth\Gth}|^{2}+2|f_{r,zz}|^{2}+4|f_{r,\Gth z}|^{2}\}d\Bx,
\]

\begin{theorem}
  \label{th:ultsimp}
The pair $(X_{\rm lin}(m(h),n(h)),\mathfrak{K}^{*}(h,\BGf))$ characterizes buckling.
\end{theorem}
\begin{proof}
We split the proof into a sequence of lemmas.
\begin{lemma}
  \label{lem:ultsimp1}
Suppose $m(h)\ge 1$ and $n(h)$ are integers satisfying (\ref{mnbound}) for all
$h\in(0,h_{0})$. Then,
There exists a constant $C(L)>0$, such
that for any $\Hat{\Bf}\in\bb{C}^{3}$ we have
\begin{equation}
\label{K0<K^ast}
\mathfrak{K}_{0}(h,\Bu_{h})\leq C(L)\mathfrak{K}^{*}(h,\Bu_{h}),
\end{equation}
where $\Bu_{h}=\CU_{m(h),n(h)}(\Hat{\Bf})\in X_{\rm lin}(m(h),n(h))$.
\end{lemma}
\begin{proof}
  For simplicity denote $\|f\|=\|f\|_{L^2(\CC_{0})}.$ If $n=0$,
  then
\[
\mathfrak{K}^{*}(h,\Bu)=\mathfrak{K}_{0}(h,\Bu)+\frac{|\Hat{f}_{\Gth}|^{2}}{|\Hat{f}_{r}|^{2}},
\]
from which (\ref{K0<K^ast}) follows. Now, let us assume that $n\geq 1$. For
each $\Bf=\CF_{m,n}(\Hat{\Bf})$ we have
\begin{equation}
\label{formulaQ1*}
Q_{1}^{*}(\Bf)=(\GL+2)(\Hat{m}^{2}+n^{2})^{2}|\Hat{f}_r|^{2}.
\end{equation}
We also have that
\[
|Q_{1}(\Bf)-Q_{1}^{*}(\Bf)|\le(\GL+2)(2(\|f_{r,zz}\|+\|f_{r,\theta\theta}\|)\|f_{\theta,\theta}\|+
\|f_{\theta,\theta}\|^2)+4\|f_{\theta,z}\|\|f_{r,\theta z}\|+\|f_{\theta,z}\|^{2},
\]
Computing the norms in terms of the Fourier coefficients we have
\begin{equation}
\label{Q1Q1star}
|Q_{1}(\Bf)-Q_{1}^{*}(\Bf)|\leq 6(\GL+2)(n+1)(\Hat{m}^{2}+n^2)
|\Hat{f}_\theta|(|\Hat{f}_r|+|\Hat{f}_\theta|).
\end{equation}
Consider now 2 cases.

\textbf{Case 1.} $|\Hat{f}_\Gth|<2|\Hat{f}_r|.$\\
In this case we have according to (\ref{Q1Q1star}) and (\ref{formulaQ1*}) that
\[
\frac{|Q_{1}(\Bf)-Q_{1}^{*}(\Bf)|}{Q_{1}^{*}(\Bf)}\leq\frac{36n}{\Hat{m}^2+n^2}\leq 36,
\]
and the inequality (\ref{K0<K^ast}) follows.

\textbf{Case 2.} $|\Hat{f}_\Gth|\ge2|\Hat{f}_r|.$\\
Observe that
$$\|f_{\Gth,\Gth}+f_{r}\|\geq \|f_{\Gth,\Gth}\|-\|f_{r}\|\geq n|\Hat{f}_{\Gth}|-
\frac{|\Hat{f}_\Gth|}{2}\geq \frac{n}{2}|\Hat{f}_{\Gth}|,$$
thus
$$Q_{0}(\Bf)\geq 2\|f_{\Gth,\Gth}+f_{r}\|^{2}\geq\frac{n^2}{2}|\Hat{f}_\Gth|^2.$$
Dividing (\ref{Q1Q1star}) by this inequality we obtain
\[
\frac{|Q_{1}(\Bf)-Q_{1}^{*}(\Bf)|}{Q_{0}(\Bf)}\leq
\frac{12(\GL+2)(\Hat{m}^2+n^2)(n+1)}{n^{2}}\Bigg(1+\frac{|\Hat{f}_r|}{|\Hat{f}_\Gth|}\Bigg)\le
36(\GL+2)(\Hat{m}^2+n^{2}).
\]
Thus
$$|\mathfrak{K}_{0}(h,\Bu)-\mathfrak{K}^{*}(h,\Bu)|=\frac{h^2}{12}\frac{|Q_{1}(\Bf)-Q_{1}^{*}(\Bf)|}{Q_{0}(\Bf)}\frac{Q_{0}(\Bf)}{B(\Bf)}\leq3h^{2}(\GL+2)(\Hat{m}^2+n^{2})\mathfrak{K}^{*}(h,\Bu).$$
Recalling that $m=m(h)$ and $n=n(h)$ satisfy (\ref{mnbound}) we conclude that
(\ref{K0<K^ast}) holds.
\end{proof}

\begin{lemma}
  \label{lem:ultsimp2}
Suppose $\Bu_{h}\in X_{\rm lin}(m(h),n(h))$ is such that there is a constant
$C_{0}$ independent of $h$, such that $\mathfrak{K}_{0}(h,\Bu_{h})\le
C_{0}h$. Then there exists $C_{1}>0$ depending only on $L$, $\SFL_{0}$ and $C_{0}$, such that
\begin{equation}
  \label{nofth}
|\mathfrak{K}^{*}(h,\Bu_{h})-\mathfrak{K}_{0}(h,\Bu_{h})|\le C_{1}h^{1/4}\mathfrak{K}_{0}(h,\Bu_{h}).
\end{equation}
\end{lemma}
\begin{proof}
Before we start the proof we remark that any buckling mode would satisfy all
the conditions of this Lemma. Let $\Hat{\Bf}_{h}\in\bb{C}^{3}$ be such that
$\Bu_{h}=\CU_{m(h),n(h)}(\Hat{\Bf}_{h})$. We also define
$\Bf_{h}=\CF_{m(h),n(h)}(\Hat{\Bf}_{h})$. We will suppress the explicit
dependence on $h$ in our notation below, and use $m$, $n$, $\Bf$ and $\Hat{\Bf}$
instead of $m(h)$, $n(h)$, $\Bf_{h}$ and $\Hat{\Bf}_{h}$, respectively.

We start with the application of
Lemma~\ref{lem:KTI} to $\Bu_{h}\in\CA_{h}$.
We compute
\[
\|u_{\Gth,z}\|^{2}_{L^{2}(\CC_{h})}=h\Hat{m}^{2}|\Hat{f}_{\Gth}|^{2}+
\frac{h^{3}\Hat{m}^{2}}{12}|\Hat{f}_{\Gth}-in\Hat{f}_{r}|^{2}\ge h\Hat{m}^{2}|\Hat{f}_{\Gth}|^{2}.
\]
Then, according to Lemma~\ref{lem:KTI} we get
\begin{equation}
\label{ineqlemma3.2}
h\Hat{m}^2\|f_{\Gth}\|^2\leq 2\|e(\Bu_{h})\|_{L^2(\CC_h)}(\|u_r\|_{L^2(\CC_h)}+\|e(\Bu_{h})\|_{L^2(\CC_h)}).
\end{equation}
By coercivity of $\SFL_{0}$ and the assumption of the
Lemma we have
\[
C_{0}h\geq\mathfrak{K}_1(h,\Bu_{h})\geq\nth{\Ga_{\SFL_{0}}}
\frac{\|e(\Bu_{h})\|_{L^2(\CC_h)}^2}{\|u_{r,z}\|_{L^2(\CC_h)}^2},
\]
Thus, there is a constant $C=\Ga_{\SFL_{0}}C_{0}$ such that
$$\|e(\Bu_{h})\|_{L^2(\CC_h)}^2\leq Ch\|\psi_{r,z}\|_{L^2(\CC_h)}^2=
Ch^2\Hat{m}^2|\Hat{f}_r|^2.$$
Using this inequality to eliminate $\|e(\Bu_{h})\|_{L^2(\CC_h)}$ from the
\rhs\ of (\ref{ineqlemma3.2}) we obtain
$$|\Hat{f}_\Gth|^2\leq C\Bigg(\frac{\sqrt{h}}{\Hat{m}}+h\Bigg)|\Hat{f}_r|^2
\leq C|\Hat{f}_r|^2\sqrt h$$
for a possibly different constant $C$.
Using this inequality to eliminate $|\Hat{f}_\Gth|$ from the \rhs\ of
(\ref{Q1Q1star}) we obtain
\[
|Q_{1}(\Bf)-Q_{1}^{*}(\Bf)|\le
Ch^{1/4}(n+1)(\Hat{m}^{2}+n^{2})|\Hat{f}_r|^2,
\]
Recalling the formula (\ref{formulaQ1*}) for $Q_{1}^{*}(\Bf)$ we get
\[
\frac{|Q_{1}(\Bf)-Q_{1}^{*}(\Bf)|}{|Q_{1}^{*}(\Bf)|}\leq
\frac{C_{1}h^{1/4}(n+1)}{\Hat{m}^2+n^2}\leq C_{2}h^{1/4}.
\]
It is now clear that
\begin{equation}
  \label{wrongK}
|\mathfrak{K}^{*}(h,\Bu_{h})-\mathfrak{K}_{0}(h,\Bu_{h})|\leq
\frac{Ch^{2+\frac{1}{4}}Q_{1}^{*}(\Bf)}{12B(\Bf)}\leq Ch^{1/4}\mathfrak{K}^{*}(h,\Bu_{h}).
\end{equation}
We also have
\[
\mathfrak{K}^{*}(h,\Bu_{h})\le\mathfrak{K}_{0}(h,\Bu_{h})+|\mathfrak{K}^{*}(h,\Bu_{h})-\mathfrak{K}_{0}(h,\Bu_{h})|.
\]
Therefore, (\ref{wrongK}) also implies (\ref{nofth}).
\end{proof}
We are now ready to prove the properties (a)--(c) in
Definition~\ref{def:equiv}. Let
$$\Hat{\Gl}^{*}(h)=\inf_{\BGf\in X_{\rm lin}(m(h),n(h))}\mathfrak{K}^{*}(h,\BGf).$$
Let $\BGy_{h}\in X_{\rm lin}(m(h),n(h))$ be a buckling mode, whose existence
is guaranteed by the Corollary~\ref{cor:main}. Then by
Lemma~\ref{lem:ultsimp2} we have
\[
\lim_{h\to 0}\frac{\mathfrak{K}^{*}(h,\BGy_{h})}{\Gl(h)}=1.
\]
Part (b) of Definition~\ref{def:equiv} is proved. In particular, we obtain
\begin{equation}
\label{ineqL1*(m,n)L1}
\lims_{h\to 0}\frac{\Hat{\Gl}^{*}(h)}{\Gl(h)}\leq 1.
\end{equation}
Now, let $\BGf_h\in X_{\rm lin}(m(h),n(h))$ be such that
\[
\lim_{h\to 0}\frac{\mathfrak{K}^{*}(h,\BGf_h)}{\Hat{\Gl}^{*}(h)}=1.
\]
Then by (\ref{ineqL1*(m,n)L1}) we have $\mathfrak{K}^{*}(h,\BGf_h)\leq Ch$ for
some $C>0,$ and thus, by Lemma~\ref{lem:ultsimp1} we have
$\mathfrak{K}_0(h,\BGf_h)\leq C(L)h$. The inequality (\ref{wrongK}) then
implies that
\begin{equation}
  \label{partc}
\lim_{h\to 0}\frac{\mathfrak{K}_0(h,\BGf_h)}{\Hat{\Gl}^{*}(h)}=1.
\end{equation}
Therefore,
$$\limsup_{h\to 0}\frac{\Hat{\Gl}_1(h)}{\Hat{\Gl}^{*}(h)}\leq 1,$$
which together with (\ref{ineqL1*(m,n)L1}) implies the validity of part (a) of
Definition~\ref{def:equiv}. In particular, this implies that
\[
\lim_{h\to 0}\frac{\mathfrak{K}_0(h,\BGf_h)}{\Gl(h)}=1.
\]
Hence, $\BGf_{h}$ must be a buckling mode, since the pair
$(X_{\rm lin}(m(h),n(h)))$ characterizes buckling.  This proves part (c)
Definition~\ref{def:equiv}.
\end{proof}

\subsection{Explicit formulas for buckling load and buckling mode}
\label{sub:blbm}
In this section we solve the minimization problem
$$\inf_{\Bu\in X_{\rm lin}(m,n)}\mathfrak{K}^{*}(h,\Bu),$$
for any pair of integers $m\geq1$ and $n\neq 0$ satisfying (\ref{mnbound}).
Let  $\Bu\in X_{\rm lin}(m,n)$ be given by (\ref{ulin}).
Then $f_{r}(\Gth,0)=f_{\Gth}(\Gth,0)=f_{r}(\Gth,L)=f_{\Gth}(\Gth,L)=0,$ and
\[
\mathfrak{K}^{*}(h,\Bu)=\mu\frac{Q_{0}(\Bf)+\frac{h^{2}}{12}Q_{1}^{*}(\Bf)}{B(\Bf)},\qquad
B(\Bf)=\int_{\CC_{0}}|f_{r,z}|^{2}dzd\Gth,
\]
where
\[
Q_{0}(\Bf)=\int_{\CC_{0}}\{\GL|f_{\Gth,\Gth}+f_{z,z}+f_{r}|^{2}+2|f_{\Gth,\Gth}+f_{r}|^{2}
+2|f_{z,z}|^{2}+|f_{z,\Gth}+f_{\Gth,z}|^{2}\}d\Bx,
\]
\[
Q_{1}^{*}(\Bf)=\int_{\CC_{0}}\{\GL|f_{r,zz}+f_{r,\Gth\Gth}|^{2}
+2|f_{r,\Gth\Gth}|^{2}+2|f_{r,zz}|^{2}+|2f_{r,\Gth z}|^{2}\}d\Bx.
\]
When $\Bf(\Gth,z)$ is such that $\Bu\in X_{\rm lim}(m,n)$ we obtain
\[
  Q_{0}(\Hat{\BGf})=\GL|in\Hat{\phi}_{\Gth}-\Hat{m}\Hat{\phi}_{z}+\Hat{\phi}_{r}|^{2}+
2|in\Hat{\phi}_{\Gth}+\Hat{\phi}_{r}|^{2}+2\Hat{m}^{2}|\Hat{\phi}_{z}|^{2}
+|in\Hat{\phi}_{z}+\Hat{m}\Hat{\phi}_{\Gth}|^{2},
\]
\[
Q_{1}^{*}(\Hat{\BGf})=(\GL+2)(\Hat{m}^{2}+n^{2})^{2}|\Hat{\phi}_{r}|^{2}.
\]
The minimum of $Q_{0}(\Hat{\BGf})$ in
$(\Hat{\phi}_{\Gth},\Hat{\phi}_{z})$ is achieved at
\begin{equation}
  \label{phithetaz}
  \begin{cases}
    \Hat{\phi}_{\Gth}^{*}=in\Hat{\phi}_{r}\dfrac{(3\GL+4)\Hat{m}^{2}+(\GL+2)n^{2}}
{(\GL+2)(n^{2}+\Hat{m}^{2})^{2}},\\[2ex]
\Hat{\phi}_{z}^{*}=\Hat{m}\Hat{\phi}_{r}\dfrac{\GL\Hat{m}^{2}-(\GL+2)n^{2}}
{(\GL+2)(n^{2}+\Hat{m}^{2})^{2}}.
  \end{cases}
\end{equation}
Substituting these values back into the quadratic form $Q_{0}$ we obtain
\[
Q_{0}=\frac{4|\Hat{\phi}_{r}|^{2}\Hat{m}^{4}(\GL+1)}{(\GL+2)(n^{2}+\Hat{m}^{2})^{2}},
\]
and hence
\[
\Hat{\Gl}^{*}(h;m,n)=\frac{4\Hat{m}^{2}(\GL+1)}
{(\GL+2)(n^{2}+\Hat{m}^{2})^{2}}+
\frac{h^{2}(\GL+2)(\Hat{m}^{2}+n^{2})^{2}}{12\Hat{m}^{2}}.
\]
Minimizing in $(m,n)$ we obtain
\begin{equation}
  \label{Koiterbl}
\Hat{\Gl}^{*}(h)=\mu\min_{m\ge 1\atop n\ge 0}\left\{\frac{4\Hat{m}^{2}(\GL+1)}
{(\GL+2)(n^{2}+\Hat{m}^{2})^{2}}+
\frac{h^{2}(\GL+2)(\Hat{m}^{2}+n^{2})^{2}}{12\Hat{m}^{2}}\right\}=2\mu h\sqrt{\frac{\GL+1}{3}},
\end{equation}
achieved at the Koiter's circle:
\begin{equation}
  \label{KoiterC}
h(\GL+2)(n^{2}+\Hat{m}^{2})^{2}=4\Hat{m}^{2}\sqrt{3(\GL+1)}.
\end{equation}
We see how this equation implies our bounds $\Hat{m}(h)^{2}h\le C$ and
$hn(h)^{4}\le C\Hat{m}(h)^{2}$.
Hence, for any $m=0,1,\ldots,M(h)$ we define
\begin{equation}
  \label{Koitern}
n(m)=\left[\sqrt{2\Hat{m}\frac{\sqrt[4]{3(\GL+1)}}{\sqrt{h(\GL+2)}}-\Hat{m}^{2}}\right],
\end{equation}
where
\begin{equation}
  \label{mhmax}
M(h)=\left[\frac{2L}{\pi}\frac{\sqrt[4]{3(\GL+1)}}{\sqrt{h(\GL+2)}}\right].
\end{equation}
The buckling modes can then be labeled by the wave number $m=0,1,\ldots,M(h)$
and given by
\[
\begin{cases}
\phi_{r}=\sin(\Hat{m}z)\cos(n(m)\Gth),\\[2ex]
\phi_{\Gth}=-hn(m)\dfrac{(3\GL+4)\Hat{m}^{2}+(\GL+2)n^{2}}
{4\Hat{m}^{2}\sqrt{3(\GL+1)}}\sin(n(m)\Gth)\sin(\Hat{m}z),\\[2ex]
\phi_{z}=h\dfrac{\GL\Hat{m}^{2}-(\GL+2)n(m)^{2}}
{4\Hat{m}\sqrt{3(\GL+1)}}\cos(\Hat{m}z)\cos(n(m)\Gth).
\end{cases}
\]

\begin{figure}[t]
 \centering
 \includegraphics[scale=0.5]{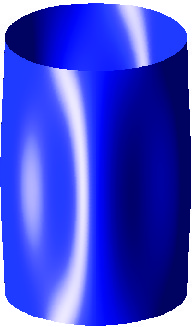}\hspace{10ex}
 \includegraphics[scale=0.5]{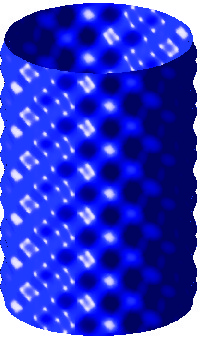}\hspace{10ex}
 \includegraphics[scale=0.5]{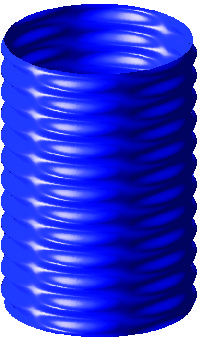}
 \caption{Buckling modes corresponding, left to right, to $m=1$, $m=M(h)/2$
   and $m=M(h)$ on the Koiter's circle.}
 \label{fig:mode1}
\end{figure}
The figure of the buckling mode corresponding to $m=1$ is shown in Figure~\ref{fig:mode1}.

\section{Fixed bottom boundary conditions}
\setcounter{equation}{0}
\label{sec:altbc}
If the \bc s (\ref{BCfixed}) are imposed, then we can no longer work with a
single Fourier mode space $X(m,n)$, since it has a zero intersection with the
space $W_{h}$ defined by (\ref{Wh}). Hence most of the analysis in
Section~\ref{sec:perfect} cannot be done for the fixed bottom boundary
conditions. However, we can still compute the buckling load and exhibit
buckling modes by modifying the explicit formulas (\ref{phithetaz}) for the
buckling modes for the \bc s (\ref{BCaverage}). According to
Remark~\ref{rem:K0} the pair $(\CA_{h}\cap W_{h},\mathfrak{K}_{0})$
characterizes buckling for the \bc s (\ref{BCfixed}). It is therefore clear that
\[
\Tld{\Gl}_{0}(h)=\inf_{\BGf\in W_{h}\cap\CA_{h}}\mathfrak{K}_{0}(h,\BGf)\ge
\Hat{\Gl}_{0}(h)=\inf_{\BGf\in\CA_{h}}\mathfrak{K}_{0}(h,\BGf).
\]
Therefore,
\[
\limi_{h\to 0}\frac{\Tld{\Gl}_{0}(h)}{\Gl(h)}\ge 1.
\]
If we find a specific test function $\Bu_{h}\in W_{h}\cap\CA_{h}$ such that
\[
\lim_{h\to 0}\frac{\mathfrak{K}_{0}(h,\Bu_{h})}{\Gl(h)}=1
\]
then
\[
1=\lim_{h\to 0}\frac{\mathfrak{K}_{0}(h,\Bu_{h})}{\Gl(h)}\ge
\limi_{h\to 0}\frac{\Tld{\Gl}_{0}(h)}{\Gl(h)}\ge 1.
\]
Which proves that $\Bu_{h}\in W_{h}$ is a buckling mode and
\[
\lim_{h\to 0}\frac{\Tld{\Gl}_{0}(h)}{\Gl(h)}=1.
\]
The idea is to look for the buckling mode in the space $X_{\rm lin}^{0}(n)$ of
all functions of the form
\begin{equation}
  \label{X0lin}
  \begin{cases}
  u_{r}=\re(\phi_{r}(z)e^{in\Gth}),& \phi_{r}\in W_{0}^{1,2}([0,L];\bb{C}),\ \phi_{r}'(0)=0,\\
  u_{\Gth}=\re(e^{in\Gth}(r\phi_{\Gth}(z)-(1-r)in\phi_{r}(z))),& \phi_{\Gth}\in W_{0}^{1,2}([0,L];\bb{C}),\\
  u_{z}=\re((\phi_{z}(z)-(1-r)\phi'_{r}(z))e^{in\Gth}),& \phi_{z}\in W^{1,2}([0,L];\bb{C}),\ \phi_{z}(0)=0.
  \end{cases}
\end{equation}
For any $\Bu\in X_{\rm lin}^{0}(n(h))$ we have
\[
\mathfrak{K}_{0}(h,\Bu)=\mu\frac{Q_{0}^{0}(\BGf)+\frac{h^{2}}{12}Q_{1}^{0}(\BGf)}{B_{0}(\BGf)},
\]
where
\[
Q_{0}^{0}(\BGf)=\int_{0}^{L}\{\GL|in\phi_{\Gth}+\phi'_{z}+\phi_{r}|^{2}+2|in\phi_{\Gth}+\phi_{r}|^{2}
+2|\phi'_{z}|^{2}+|in\phi_{z}+\phi'_{\Gth}|^{2}\}dz,
\]
\[
Q_{1}^{0}(\BGf)=\int_{0}^{L}\{\GL|\phi''_{r}-n^{2}\phi_{r}-in\phi_{\Gth}|^{2}
+2|n^{2}\phi_{r}+in\phi_{\Gth}|^{2}+2|\phi''_{r}|^{2}+|\phi'_{\Gth,}-2in\phi'_{r}|^{2}\}dz,
\]
\[
B_{0}(\BGf)=\int_{0}^{L}|\phi'_{r}|^{2}dz.
\]
Even though the fixed bottom \bc s prevent the problem to be diagonalized in
the Fourier space, it is still useful to represent $\BGf(z)$ in the form of
Fourier series
\[
\begin{cases}
\displaystyle\phi_{r}(z)=\sum_{m=1}^{\infty}\Hat{\phi}_{r}(m)\sin(\Hat{m}z),\\
\displaystyle\phi_{\Gth}(z)=\sum_{m=1}^{\infty}\Hat{\phi}_{\Gth}(m)\sin(\Hat{m}z),\\
\displaystyle\phi_{z}(z)=\sum_{m=0}^{\infty}\Hat{\phi}_{z}(m)\cos(\Hat{m}z).
\end{cases}
\]
The boundary condition (\ref{BCfixed}) translate via (\ref{X0lin}) into the
additional constraints
\begin{equation}
  \label{constraints}
\sum_{m=1}^{\infty}m\Hat{\phi}_{r}(m)=0,\qquad\sum_{m=0}^{\infty}\Hat{\phi}_{z}(m)=0.
\end{equation}
In terms of Fourier coefficients we have
\[
Q_{0}^{0}(\BGf)=\sum_{m=0}^{\infty}Q_{m}^{0}(\Hat{\BGf}),\quad
Q_{1}^{0}(\BGf)=\sum_{m=1}^{\infty}Q_{m}^{1}(\Hat{\BGf}),\quad
B_{0}(\BGf)=\sum_{m=1}^{\infty}B_{m}^{0}(\Hat{\BGf}),
\]
where
\[
Q_{m}^{0}(\Hat{\BGf})=\GL|in\Hat{\phi}_{\Gth}-\Hat{m}\Hat{\phi}_{z}+\Hat{\phi}_{r}|^{2}+
2|in\Hat{\phi}_{\Gth}+\Hat{\phi}_{r}|^{2}+2\Hat{m}^{2}|\Hat{\phi}_{z}|^{2}
+|in\Hat{\phi}_{z}+\Hat{m}\Hat{\phi}_{\Gth}|^{2},
\]
\[
Q_{m}^{1}(\Hat{\BGf})=\GL|\Hat{m}^{2}\Hat{\phi}_{r}+n^{2}\Hat{\phi}_{r}+in\Hat{\phi}_{\Gth}|^{2}
+2|n^{2}\Hat{\phi}_{r}+in\Hat{\phi}_{\Gth}|^{2}+2\Hat{m}^{4}|\Hat{\phi}_{r}|^{2}+\Hat{m}^{2}|\Hat{\phi}_{\Gth,}-2in\Hat{\phi}_{r}|^{2}.
\]
\[
B_{m}^{0}(\Hat{\BGf})=\Hat{m}^{2}|\Hat{\phi}_{r}|^{2}.
\]
The fixed bottom \bc s do not place any additional constraints on the Fourier
coefficients of $\phi_{\Gth}$. Therefore, we can minimize
$Q_{m}^{0}+h^{2}Q_{m}^{1}/12$ in $\Hat{\phi}_{\Gth}$ to obtain an explicit
expression of $\Hat{\phi}_{\Gth}(m)$ in terms of $\Hat{\phi}_{r}(m)$ and
$\Hat{\phi}_{z}(m)$. However, we may simplify the algebra by
recalling that the functional $\mathfrak{K}^{*}$ could be used to compute the
buckling load. We therefore determine the relation between
$\Hat{\phi}_{\Gth}(m)$ and $\Hat{\phi}_{r}(m)$, and $\Hat{\phi}_{z}(m)$ by
minimizing $Q_{m}^{0}$ in $\Hat{\phi}_{\Gth}$. We obtain
\begin{equation}
  \label{phithfixedbc}
\Hat{\phi}_{\Gth}=in\frac{(\GL+2)\Hat{\phi}_{r}-(\GL+1)\Hat{m}\Hat{\phi}_{z}}
{(\GL+2)n^{2}+\Hat{m}^{2}}.
\end{equation}
We now cook-up a test function based on (\ref{phithetaz}).

Let $m=m(h)$ be such that
\begin{equation}
  \label{m(h)}
\lim_{h\to 0}m(h)=\infty,\qquad\lim_{h\to 0}m(h)\sqrt{h}=0.
\end{equation}
Let $n=n(h)=n(m(h))$ be given by (\ref{Koitern}). We remark that
under our assumptions $n(h)\gg m(h)$. Let
\[
\phi_{r}=\left(\frac{\sin(\Hat{m}z)}{\Hat{m}}-
\frac{\sin(\Hat{m+2}z)}{\Hat{m+1}}\right)\cos(n\Gth).
\]
This function satisfies all the required \bc s in (\ref{X0lin}). It's
$z$-derivative also vanishes at the top of the cylinder, even though we do not
require it. If we define $\Hat{\phi}_{z}(m)$ by (\ref{phithetaz}) then the
resulting function $\phi_{z}$ will not vanish exactly at the bottom of the
cylindrical shell. That is why we modify
(\ref{phithetaz}) as follows
\begin{equation}
  \label{phizfixedbc}
  \phi_{z}=T(m,n)(\cos(\Hat{m+2}z)-\cos(\Hat{m}z))\cos(n\Gth),
\end{equation}
where
\[
T(m,n)=\frac{(\GL+2)n^{2}-\GL\Hat{m}^{2}}{(\GL+2)(n^{2}+\Hat{m}^{2})^{2}}.
\]
Once again we observe that $\phi_{z}$ vanishes not only at the bottom
boundary, but also at the top, accommodating even pure displacement \bc s on top
and bottom edges. We may simplify our test function if we retain only the necessary
asymptotics as $h\to 0$ in (\ref{phithfixedbc}) and (\ref{phizfixedbc}):
\begin{equation}
  \label{fixedbcmode}
  \begin{cases}
\phi_{r}=\left(\dfrac{\sin(\Hat{m}z)}{\Hat{m}}-
\dfrac{\sin(\Hat{m+2}z)}{\Hat{m+2}}\right)\cos(n\Gth),\\
\phi_{z}=\nth{n^{2}}(\cos(\Hat{m+2}z)-\cos(\Hat{m}z))\cos(n\Gth),\\
\phi_{\Gth}=-\nth{n}(\Gg(m,n)\sin(\Hat{m}z)-\Gg(m+2,n)\sin(\Hat{m+2}z))\sin(n\Gth),
\end{cases}
\end{equation}
\[
\Gg(m,n)=\nth{\Hat{m}}+\frac{\GL\Hat{m}}{(\GL+2)n^{2}}.
\]
Substituting this into $\mathfrak{K}_{0}$ we obtain
\[
\lim_{h\to 0}\frac{\mathfrak{K}_{0}(h,\BGf)}{h}=
\lim_{h\to 0}\frac{\mu}{2}\sqrt{\frac{\GL+1}{3}}\left(2+
\frac{(\Hat{m+2})^{2}}{\Hat{m}^{2}}+\frac{\Hat{m}^{2}}{(\Hat{m+2})^{2}}\right).
\]
We conclude that
\[
\lim_{h\to 0}\frac{\mathfrak{K}_{0}(h,\BGf)}{h}=\lim_{h\to 0}\frac{\Gl(h)}{h}=
2\mu\sqrt{\frac{\GL+1}{3}},
\]
since
\[
\lim_{h\to 0}\frac{(\Hat{m+2})^{2}}{\Hat{m}^{2}}=1.
\]

\begin{figure}[t]
 \centering
\includegraphics[scale=0.5]{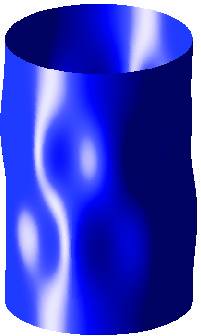}\hspace{8ex}
 \includegraphics[scale=0.5]{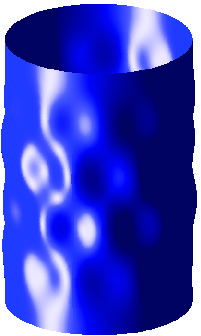}\hspace{8ex}
\includegraphics[scale=0.5]{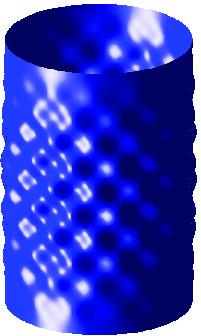}
 \caption{Buckling mode (\ref{X0lin}), (\ref{fixedbcmode}) corresponding, left to right, to
   $m(h)\sim h^{-1/8}$, $h^{-1/4}$ and $h^{-3/8}$.}
 \label{fig:fixedbc}
\end{figure}
Thus, the test functions (\ref{fixedbcmode}) are buckling modes for any $m(h)$
satisfying (\ref{m(h)}). Figure~\ref{fig:fixedbc} shows the buckling mode
(\ref{fixedbcmode}) for
\[
\Hat{m(h)}=\left(\frac{4\sqrt{3(\GL+1)}}{h(\GL+2)}\right)^{\Ga},\quad\Ga=1/8,\
1/4, 3/8.
\]

\section{Discussion}
\label{sec:conc}
The key observation in our analysis is that for the test functions $\BU^{h}$
(\ref{ansatz0}) we have
\[
\mathfrak{S}_{h}(\BU^{h})=O(\|e(\BU^{h})\|^{2})=O(K(V_{h})\|\Grad\BU^{h}\|^{2}).
\] 
However, the asymptotics of the destabilizing compressiveness term
\[
\mathfrak{C}_{h}(\BU^{h})=\int_{\GO_{h}}(\BGs_{h},(\Grad\BU^{h})^{T}\Grad\BU^{h})d\Bx
\]
depends strongly on the structure of the tensor
\[
\BGs^{0}(\Gth,z)=\lim_{h\to 0}\BGs_{h}(r,\Gth,z),
\]
We saw that for the perfect cylinder $\BGs^{0}=\BGs_{h}$ is given by
(\ref{perfectstr}) and hence
\[
\mathfrak{C}_{h}(\BU^{h})=O\left(\frac{\|e(\BU^{h})\|^{2}}{h}\right),
\]
If we assume that
\begin{equation}
  \label{exp}
\BGs_{h}(r,\Gth,z;h)=\BGs^{0}(\Gth,z)+h\BGt(\Gth,z)+(r-1)\BGs^{1}(\Gth,z)+O(h^{2}).  
\end{equation}
Substituting it into the equilibrium equation $\Div\BGs^{h}=\Bzr$ and passing
to the limit as $h\to 0$, we obtain
\begin{equation}
  \label{limequil}
  \begin{cases}
  \Gs^{1}_{rr}+\Gs^{0}_{r\Gth,\Gth}+\Gs^{0}_{rr}-\Gs^{0}_{\Gth\Gth}+\Gs^{0}_{rz,z}=0,\\
\Gs^{1}_{r\Gth}+\Gs^{0}_{\Gth\Gth,\Gth}+2\Gs^{0}_{r\Gth}+\Gs^{0}_{\Gth z,z}=0,\\
\Gs^{1}_{rz}+\Gs^{0}_{\Gth z,\Gth}+\Gs^{0}_{rz}+\Gs^{0}_{zz,z}=0.
  \end{cases}
\end{equation}
The traction-free \bc on the lateral boundary of the shell $r=1\pm h/2$ implies that
\[
\BGs^{0}(\Gth,z)\Be_{r}=\BGs^{1}(\Gth,z)\Be_{r}=\BGt(\Gth,z)\Be_{r}=\Bzr
\]
for all $(\Gth,z)\in\bb{T}\times(0,L)$. Substituting these equations into
(\ref{limequil}) we obtain
\[
\Gs_{rr}^{0}=0,\quad\Gs_{r\Gth}^{0}=0,\quad\Gs_{rz}^{0}=0,\quad\Gs^{0}_{\Gth\Gth}=0.
\]
We also obtain $\Gs^{0}_{\Gth z,z}=0$ and $\Gs^{0}_{\Gth z,\Gth}+\Gs^{0}_{zz,z}=0$.
Solving these equations results in the following form for $\BGs^{0}$:
\begin{equation}
  \label{linstr}
\BGs^{0}(\Gth,z)=\left[
  \begin{array}{ccc}
    0 & 0 & 0\\
    0 & 0 & s(\Gth)\\
    0 & s(\Gth) & t(\Gth)-zs'(\Gth)
  \end{array}
\right].
\end{equation}
for some functions $s(\Gth)$ and $t(\Gth)$. For generic choices of $s(\Gth)$
and $t(\Gth)$ we obtain
\[
\mathfrak{C}_{h}(\BU^{h})=O\left(\frac{\|e(\BU^{h})\|^{2}}{h^{5/4}}\right),
\]
resulting in $\Gl(h)=O(h^{5/4})$. One might conjecture that the imperfections
of load can produce a non-homogeneous trivial branch leading to $\BGs^{0}$
given by (\ref{linstr}) and hence to the dramatic change in the asymptotic
behavior of $\Gl(h)$. 

If we disregard the calculations leading to (\ref{linstr}) and assume for a
moment that $\Gs^{0}_{\Gth\Gth}\not=0$ then
\[
\mathfrak{C}_{h}(\BU^{h})=O\left(\frac{\|e(\BU^{h})\|^{2}}{h^{3/2}}\right)=
O\left(\frac{\|e(\BU^{h})\|^{2}}{K(V_{h})}\right).
\]
It may be conjectured that imperfection of shape can be mathematically
described by such tensor $\BGs^{0}$. In this case the critical load has the
asymptotics $\Gl(h)\sim K(V_{h})=O(h^{3/2})$. We note that the exponents
$5/4=1.25$ and $3/2=1.5$ are close the upper and lower limits of
experimentally determined behavior of the buckling load \cite{call2000,hlp03}.

Observe that $\mathfrak{C}_{h}(\BGf)$
cannot be larger than $\|e(\BGf)\|^{2}/K(V_{h})$. Therefore, if the predicted
buckling load $\Gl(h)\sim K(V_{h})$ (Euler buckling in the terminology of
\cite{grtr07}) then the imperfections of load and shape will have
negligible effect on the buckling load as in the case of straight solid struts
and flat plates. 

\medskip

\noindent\textbf{Acknowledgments.}  We are grateful to Lev Truskinovsky for
reading the entire manuscript and suggesting many improvements in the
exposition. This material is based upon work supported by the National Science
Foundation under Grants No. 1008092.

\appendix

\section{Proof of Theorem~\ref{th:KI}}
\setcounter{equation}{0}
\label{app:KI}
The proof of Theorem~\ref{th:KI} is quite involved and will be split into
several relatively simple steps.

\subsection{Zero \bc on a rectangle}
\label{sub:zbc-rect}
For any vector field $\BU=(u,v)$ on $\GO=[0,h]\times[0,L]$ and any
$\Ga\in[-1,1]$ we define
\[
\BG_{\Ga}=\mat{u_{x}}{u_{y}}{v_{x}}{v_{y}+\Ga u},\qquad
\Be_{\Ga}=\hf(\BG_{\Ga}+\BG_{\Ga}^{T})=\mat{u_{x}}{\hf(u_{y}+v_{x})}{\hf(u_{y}+v_{x})}{v_{y}+\Ga u}.
\]
\begin{theorem}
  \label{th:basicineq}
Suppose that the vector field $\BU=(u,v)\in C^{1}(\bra{\GO};\bb{R}^{2})$ satisfies
$u(x,0)=u(x,L)=\Bzr$. Then for
any $\Ga\in[-1,1]$, any $h\in(0,1)$ and any $L>0$
\[
\|\BG_{\Ga}\|^{2}\le100\|\Be_{\Ga}\|\left(\frac{\|u\|}{h}+\|\Be_{\Ga}\|\right).
\]
\end{theorem}
We emphasize that there are no boundary conditions imposed on $v(x,y)$.
\begin{proof}
First we prove several auxiliary lemmas.
\begin{lemma}
  \label{lem:harmon}
Suppose $w(x,y)$ is harmonic in $[0,h]\times[0,L]$, and satisfies
$w(x,0)=w(x,L)=0$. Then
\begin{equation}
  \label{hi}
\|w_{y}\|^{2}\le\frac{2\sqrt{3}}{h}\|w\|\|w_{x}\|+\|w_{x}\|^{2}.
\end{equation}
\end{lemma}
\begin{proof}
If $w(x,y)$ is harmonic and satisfies $w(x,0)=w(x,L)=0$ then it must have the
expansion
\[
w(x,y)=\sum_{n=1}^{\infty}(A_{n}e^{\frac{\pi nx}{L}}+B_{n}e^{-\frac{\pi nx}{L}})
\sin\left(\frac{\pi n y}{L}\right).
\]
Therefore,
\[
\|w\|^{2}=\frac{Lh}{2}\sum_{n=1}^{\infty}\left\{\psi\left(\frac{\pi nh}{L}\right)
\left(A_{n}^{2}e^{\frac{\pi nh}{L}}+B_{n}^{2}e^{\frac{-\pi nh}{L}}\right)+2A_{n}B_{n}\right\},\qquad
\psi(x)=\frac{\sinh(x)}{x}.
\]
In the expansion of $w_{y}$ we simply multiply $A_{n}$ and $B_{n}$ by $\pi
n/L$, while in the expansion of $w_{x}$ we multiply $A_{n}$ by $\pi n/L$ and
$B_{n}$ by $-\pi n/L$:
\[
\|w_{y}\|^{2}=\frac{Lh}{2}\sum_{n=1}^{\infty}\frac{\pi^{2}n^{2}}{L^{2}}\left\{\psi\left(\frac{\pi nh}{L}\right)
\left(A_{n}^{2}e^{\frac{\pi nh}{L}}+B_{n}^{2}e^{\frac{-\pi nh}{L}}\right)+2A_{n}B_{n}\right\},
\]
\[
\|w_{x}\|^{2}=\frac{Lh}{2}\sum_{n=1}^{\infty}\frac{\pi^{2}n^{2}}{L^{2}}\left\{\psi\left(\frac{\pi nh}{L}\right)
\left(A_{n}^{2}e^{\frac{\pi nh}{L}}+B_{n}^{2}e^{\frac{-\pi nh}{L}}\right)-2A_{n}B_{n}\right\},
\]
The numbers $A_{n}$ and $B_{n}$ can be arbitrary, but such that all the series
converge. We can therefore change variables
\[
a_{n}=A_{n}e^{\frac{\pi nh}{2L}},\qquad b_{n}=B_{n}e^{-\frac{\pi nh}{2L}},\qquad\tau_{n}=\frac{\pi nh}{L}
\]
Then
\[
\frac{\|w\|^{2}}{h^{2}}=\frac{Lh}{2}\sum_{n=1}^{\infty}\frac{\pi^{2}n^{2}}{\tau_{n}^{2}L^{2}}
\{(\psi(\tau_{n})-1)(a_{n}^{2}+b_{n}^{2})+(a_{n}+b_{n})^{2}\},
\]
\[
\|w_{y}\|^{2}=\frac{Lh}{2}\sum_{n=1}^{\infty}\frac{\pi^{2}n^{2}}{L^{2}}\{(\psi(\tau_{n})-1)
(a_{n}^{2}+b_{n}^{2})+(a_{n}+b_{n})^{2}\},
\]
\[
\|w_{x}\|^{2}=\frac{Lh}{2}\sum_{n=1}^{\infty}\frac{\pi^{2}n^{2}}{L^{2}}\{(\psi(\tau_{n})-1)
(a_{n}^{2}+b_{n}^{2})+(a_{n}-b_{n})^{2}\},
\]
Obviously,
\[
\|w_{y}\|^{2}-\|w_{x}\|^{2}=2Lh\sum_{n=1}^{\infty}\frac{\pi^{2}n^{2}}{L^{2}}a_{n}b_{n}\le
2Lh\sum_{n\in\CP}\frac{\pi^{2}n^{2}}{L^{2}}a_{n}b_{n},\qquad\CP=\{n\in\bb{N}:a_{n}b_{n}>0.\}.
\]
Next we estimate
\[
\frac{\|w\|^{2}}{h^{2}}\ge\frac{Lh}{2}\sum_{n\in\CP}\frac{\pi^{2}n^{2}}{\tau_{n}^{2}L^{2}}
\{(\psi(\tau_{n})-1)(a_{n}^{2}+b_{n}^{2})+(a_{n}+b_{n})^{2}\}\ge
Lh\sum_{n\in\CP}\frac{\pi^{2}n^{2}}{\tau_{n}^{2}L^{2}}(\psi(\tau_{n})+1)a_{n}b_{n}.
\]
Similarly,
\[
\|w_{x}\|^{2}\ge Lh\sum_{n\in\CP}\frac{\pi^{2}n^{2}}{L^{2}}
(\psi(\tau_{n})-1)a_{n}b_{n}.
\]
Now we have
\[
\sum_{n\in\CP}\frac{\pi^{2}n^{2}}{L^{2}}a_{n}b_{n}=
\sum_{n\in\CP}\left(\frac{\pi n}{L}\sqrt{(\psi(\tau_{n})-1)a_{n}b_{n}}\right)
\left(\frac{\pi n}{L}\sqrt{\frac{a_{n}b_{n}}{\psi(\tau_{n})-1}}\right).
\]
Applying the Cauchy-Schwartz inequality we obtain
\[
\sum_{n\in\CP}\frac{\pi^{2}n^{2}}{L^{2}}a_{n}b_{n}\le
\sqrt{\sum_{n\in\CP}\frac{\pi^{2}n^{2}}{L^{2}}
(\psi(\tau_{n})-1)a_{n}b_{n}}\sqrt{\sum_{n\in\CP}\Phi(\tau_{n})
\frac{\pi^{2}n^{2}}{\tau_{n}^{2}L^{2}}(\psi(\tau_{n})+1)a_{n}b_{n}},
\]
where
\[
\Phi(\tau)=\frac{\tau^{2}}{\psi(\tau)^{2}-1}=\frac{\tau^{4}}{\sinh^{2}(\tau)-\tau^{2}}.
\]
The function $\Phi(\tau)$ is monotone decreasing on $(0,+\infty)$, and hence,
$\Phi(\tau_{n})\le\Phi(\tau_{1})\le\Phi(0)=3$. Therefore,
\begin{equation}
  \label{shi}
\|w_{y}\|^{2}-\|w_{x}\|^{2}\le\frac{2\sqrt{\Phi(\pi h/L)}}{h}\|w\|\|w_{x}\|,
\end{equation}
and the inequality (\ref{hi}) follows.
The inequality (\ref{shi}) is sharp, since it turns into equality for
\[
w(x,y)=\cosh\left(\frac{\pi}{L}\left(x-\frac{h}{2}\right)\right)
\sin\left(\frac{\pi y}{L}\right).
\]
\end{proof}
Suppose $w(x,y)$ solves
\begin{equation}
  \label{harmon}
  \begin{cases}
  \GD w(x,y)=0,&(x,y)\in\GO\\
  w(x,y)=u(x,y),&(x,y)\in\dOm,
\end{cases}
\end{equation}
where $\GO=[0,h]\times[0,L]$. Then $\Grad w$ is the Helmholtz projection of
$\Grad u$ onto the space of the divergence-free fields in
$L^{2}(\GO;\bb{R}^{2})$.
\begin{lemma}
  \label{lem:main}
Suppose that the vector field $\BU=(u,v)\in C^{1}(\bra{\GO};\bb{R}^{2})$ satisfies
$\BU(x,0)=\BU(x,L)=\Bzr$. Let $w(x,y)$ be defined by (\ref{harmon}). Then for
any $\Ga\in[-1,1]$, any $h\in(0,1)$ and any $L>0$
\begin{equation}
  \label{mainest}
  \|\Grad u-\Grad w\|\le\left(\sqrt{2}+\nth{\pi}\right)\|\Be_{\Ga}\|,\qquad
\|u-w\|\le\frac{h}{\pi}\left(\sqrt{2}+\nth{\pi}\right)\|\Be_{\Ga}\|.
\end{equation}
\end{lemma}
\begin{proof}
This follows your note with tighter constants.
\[
\GD(u-w)=\GD u=(e_{11}-e_{22})_{x}+2(e_{12})_{y}+\Ga e_{11}.
\]
Multiplying by $u-w$ and integrating we get
\[
\|\Grad(u-w)\|^{2}=\int_{\GO}\{(e_{11}-e_{22})(u-w)_{x}+2e_{12}(u-w)_{y}+\Ga
e_{11}(w-u)\}d\Bx
\]
By Cauchy-Schwartz inequality we get
\[
\|\Grad(u-w)\|^{2}\le\|\Be_{\Ga}\|(\sqrt{2}\|\Grad(u-w)\|+|\Ga|\|u-w\|).
\]
By Poincar\'e
\[
\int_{0}^{h}|u-w|^{2}dx\le\frac{h^{2}}{\pi^{2}}\int_{0}^{h}|(u-w)_{y}|^{2}dx.
\]
Hence,
\[
\|u-w\|\le\frac{h}{\pi}\|\Grad(u-w)\|,
\]
and (\ref{mainest}) follows.
\end{proof}
We are now ready to prove the theorem. For simplicity of notation we denote
\[
K_{0}=\frac{1}{\pi}\left(\sqrt{2}+\nth{\pi}\right).
\]
By the triangle inequality and Lemma~\ref{lem:main} we get
\begin{multline*}
  \|\BG_{\Ga}\|^{2}=\|\Be_{\Ga}\|^{2}+\hf\|v_{x}-u_{y}\|^{2}=\|\Be_{\Ga}\|^{2}+
\hf\|(v_{x}+u_{y})-2(u_{y}-w_{y})-2w_{y}\|^{2}\le\\
\|\Be_{\Ga}\|^{2}+\frac{3}{2}\|u_{y}+v_{x}\|^{2}+6\|u_{y}-w_{y}\|^{2}+6\|w_{y}\|^{2}\le
(4+6\pi^{2}K_{0}^{2})\|\Be_{\Ga}\|^{2}+6\|w_{y}\|^{2}.
\end{multline*}
To estimate $\|w_{y}\|$ via Lemma~\ref{lem:harmon} we
have by the triangle inequality and Lemma~\ref{lem:main}
\[
\|w\|\le\|u\|+\|u-w\|\le\|u\|+K_{0}h\|\Be_{\Ga}\|,
\]
\[
\|w_{x}\|\le\|u_{x}\|+\|w_{x}-u_{x}\|\le(1+\pi K_{0})\|\Be_{\Ga}\|.
\]
Therefore,
\[
\|w_{y}\|^{2}\le\frac{2\sqrt{3}(1+\pi K_{0})}{h}\|u\|\|\Be_{\Ga}\|+
(1+\pi K_{0})(1+(2\sqrt{3}+\pi)K_{0})\|\Be_{\Ga}\|^{2}.
\]
Thus,
\[
\|\BG_{\Ga}\|^{2}\le\frac{12\sqrt{3}(1+\pi K_{0})}{h}\|u\|\|\Be_{\Ga}\|+K_{1}\|e_{\Ga}\|^{2},
\]
where
\[
K_{1}=4+6\pi^{2}K_{0}^{2}+6(1+\pi K_{0})(1+(2\sqrt{3}+\pi)K_{0}).
\]
Rounding the constants up to the next integer we obtain
\[
\|\BG_{\Ga}\|^{2}\le99\|\Be_{\Ga}\|^{2}+\frac{57}{h}\|u\|\|\Be_{\Ga}\|.
\]
The theorem is proved now.
\end{proof}

\subsection{Periodic \bc s on a rectangle}
\label{sub:pbc-rect}
\begin{theorem}
\label{th:pbc}
Suppose that the vector field $\BU=(u,v)\in C^{1}([0,h]\times[0,2\pi];\bb{R}^{2})$ satisfies
$u(x,0)=u(x,2\pi)$. Then there exists an absolute numerical constant $C_{0}>0$
such that for
any $\Ga\in[-1,1]$ and any $h\in(0,1)$
\[
\|\BG_{\Ga}\|^{2}\le C_{0}\|\Be_{\Ga}\|\left(\frac{\|u\|}{h}+\|\Be_{\Ga}\|\right).
\]
\end{theorem}

\begin{proof}
 For any fixed $t\in[0,2\pi]$ denote $\BU_t=(u-u(x,t), v+\alpha yu(x,t))$ and $\Omega_{t}=[0,h]\times[t,t+2\pi].$ Observe that $\BU_{t}$ satisfies zero boundary conditions on the horizontal boundary of $\Omega_{t}.$ We apply now Theorem ~\ref{th:basicineq} to the displacement $\BU_{t}$ in $\Omega_{t},$
\begin{equation}
\label{3}
\|\BG_\alpha( \BU_{t})\|^2\leq \frac{100}{h}\|e(\BG_\alpha(\BU_{t}))\|\cdot\|u_t\|+100\|e(\BG_\alpha(\BU_{t}))\|^2.
\end{equation}

Note that
$$
\BG_\alpha(\BU_t)=\begin{bmatrix}
u_x-u_x(x,t) & u_y\\
v_x+\alpha yu_x(x,t) & v_y+u\\
\end{bmatrix}=
\BG_\alpha+
\begin{bmatrix}
-u_x(x,t) & 0\\
\alpha yu_x(x,t) & 0\\
\end{bmatrix}
$$
and

$$
e(\BG_\alpha(\BU_t))=\begin{bmatrix}
u_x-u_x(x,t) & \frac{1}{2}(u_y+v_x+\alpha yu_x(x,t))\\
\frac{1}{2}(u_y+v_x+\alpha yu_x(x,t)) & v_y+u\\
\end{bmatrix}=
$$
$$
=\Be_\alpha+
\begin{bmatrix}
-u_x(x,t) & \frac{1}{2}\alpha yu_x(x,t)\\
 \frac{1}{2}\alpha yu_x(x,t) & 0\\
\end{bmatrix},
$$

thus

$$\|\BG_\alpha\|\leq \|\BG_\alpha(\BU_t)\|+(1+2\pi)\|u_x(x,t)\|,$$

$$\|e(\BG_\alpha(\BU_t))\|\leq \|\Be_\alpha\|+(1+2\pi)\|u_x(x,t)\|$$
and
$$\|u_t\|\leq \|u\|+\|u(x,t)\|.$$
Utilizing now (\ref{3}) and taking into account the last three inequalities we arrive at

$$\|\BG_\alpha\|^2\leq 2\|\BG_\alpha(\BU_t)\|^2+2(1+2\pi)^2\|u_x(x,t)\|^2\leq $$
$$\leq \frac{200}{h}\|e(\BG_\alpha(\BU_{t}))\|\cdot\|u_t\|+200\|e(\BG_\alpha(\BU_{t}))\|^2+2(1+2\pi)^2\|u_x(x,t)\|^2\leq$$
$$\leq \frac{200}{h}\big(\|\Be_\alpha\|+(1+2\pi)\|u_x(x,t)\|\big)\big(\|u\|+\|u(x,t)\|\big)+$$

\begin{equation}
\label{4}
+200\big(\|\Be_\alpha\|+(1+2\pi)\|u_x(x,t)\|\big)^2+2(1+2\pi)^2\|u_x(x,t)\|^2
\end{equation}

We complete the proof integrating (\ref{4}) in $t$ over $[0,L]$ and then estimating each summand applying the Schwartz inequality as follows
$$\int_0^L\|u_x(x,t)\|\ud t\leq \Big(L\int_0^L\|u_x(x,t)\|^2\ud t\Big)^{\frac{1}{2}}=$$
$$=\Big(L^2\int_0^L\int_0^h u_x^2(x,t)\ud x\ud t\Big)^{\frac{1}{2}}=L\|u_x\|\leq L\|\Be_\alpha\|,$$
similarly
$$\int_0^L\|u(x,t)\|\ud t\leq L\|u\|,$$
$$\int_0^L\|u_x(x,t)\|^2\ud t=L\|u_x\|^2\leq L\|\Be_\alpha\|^2,$$

$$\int_0^L\|u_x(x,t)\|\|u(x,t)\|\ud t\leq \Big(\int_0^L\|u_x(x,t)\|^2\ud t\cdot \int_0^L\|u(x,t)\|^2\ud t\Big)^{\frac{1}{2}}=$$
$$=L\|u_x\|\|u\|\leq L\|\Be_\alpha\|\|u\|.$$

\end{proof}

Let
\[
\BG_{*}=\mat{u_{x}}{u_{y}-v}{v_{x}}{v_{y}+u},\qquad
\Be_{*}=\hf(\BG_{*}+\BG_{*}^{T}).
\]
\begin{theorem}
\label{th:hard}
Suppose that the vector field $\BU=(u,v)\in C^{1}([0,h]\times[0,2\pi];\bb{R}^{2})$ satisfies
$\BU(x,0)=\BU(x,2\pi)$. Then there exist absolute numerical constants $\Gs>0$
and $C_{0}>0$ such that for
any $h\in(0,\Gs)$
\[
\|\BG_{*}\|^{2}\le C_{0}\left(\|\Be_{*}\|^{2}+\|\Be_{*}\|\frac{\|u\|}{h}+\|v\|^{2}\right).
\]
\end{theorem}
\begin{proof}
Let $\BV=(u,(1-x)v)$, and let
\[
\BG_{1}=G_{1}(\BV),\qquad\Be_{1}=\hf(\BG_{1}+\BG_{1}^{T}).
\]
We compute
\[
\BG_{*}=\BG_{1}+\mat{0}{-v}{v+xv_{x}}{xv_{y}},\qquad
\Be_{1}=\Be_{*}+\mat{0}{-\dfrac{x}{2}v_{x}}{-\dfrac{x}{2}v_{x}}{-xv_{y}}.
\]
Thus we immediately obtain that
\[
\|\BG_{*}\|^{2}\le
6(\|\BG_{1}\|^{2}+\|v\|^{2}+h^{2}(\|v_{x}\|^{2}+\|v_{y}\|^{2}).
\]
and
\begin{equation}
  \label{eest}
\|\Be_{1}\|\le\|\Be_{*}\|+h(\|v_{x}\|+\|v_{y}\|)
\end{equation}
We also estimate
\begin{equation}
  \label{gradv}
\|v_{x}\|\le\|\BG_{*}\|,\qquad\|v_{y}\|\le\|v_{y}+u\|+\|u\|\le\|\Be_{*}\|+\|u\|.
\end{equation}
Now we apply Theorem~\ref{th:pbc} to the vector field $\BV$ and $\Ga=1$, and obtain
\[
\|\BG_{*}\|^{2}\le C_{0}\left(\|\Be_{1}\|^{2}+\|\Be_{1}\|\frac{\|u\|}{h}+
\|v\|^{2}+h^{2}(\|v_{x}\|^{2}+\|v_{y}\|^{2})\right).
\]
Next we apply (\ref{eest}) to the terms containing $\|\Be_{1}\|$ and obtain
\[
\|\BG_{*}\|^{2}\le C_{0}\left(\|\Be_{*}\|^{2}+\|\Be_{*}\|\frac{\|u\|}{h}+
\|u\|(\|v_{x}\|+\|v_{y}\|)+\|v\|^{2}+h^{2}(\|v_{x}\|^{2}+\|v_{y}\|^{2})\right).
\]
Applying the inequalities (\ref{gradv}) to the terms containing $\|v_{x}\|$
and $\|v_{y}\|$ we obtain
\[
\|\BG_{*}\|^{2}\le C_{0}\left(\|\Be_{*}\|^{2}+\|\Be_{*}\|\frac{\|u\|}{h}+
\|u\|\|\BG_{*}\|+\|u\|^{2}+\|v\|^{2}+h^{2}\|\BG_{*}\|^{2}\right).
\]
When $h^{2}<1/(2C_{0})$ we get the inequality
\[
\|\BG_{*}\|^{2}\le C_{0}\left(\|\Be_{*}\|^{2}+\|\Be_{*}\|\frac{\|u\|}{h}+
\|u\|\|\BG_{*}\|+\|u\|^{2}+\|v\|^{2}\right).
\]
We also have
\[
C_{0}\|u\|\|\BG_{*}\|\le\hf\|\BG_{*}\|^{2}+\frac{C_{0}^{2}}{2}\|u\|^{2}.
\]
Thus we obtain
\begin{equation}
  \label{penult}
\|\BG_{*}\|^{2}\le C_{0}\left(\|\Be_{*}\|^{2}+\|\Be_{*}\|\frac{\|u\|}{h}+\|u\|^{2}+\|v\|^{2}\right).
\end{equation}
To finish the proof of the theorem we write $\|u\|^{2}$ using integration by
parts and periodic \bc s:
\[
\|u\|^{2}=(u,u+v_{y})+(u_{y}-v,v)+\|v\|^{2}.
\]
Thus,
\[
\|u\|^{2}\le\|u\|\|\Be_{*}\|+\|\BG_{*}\|\|v\|+\|v\|^{2}.
\]
Thus, using $2\|u\|\|\Be_{*}\|\le\|u\|^{2}+\|\Be_{*}\|^{2}$ we obtain
\begin{equation}
  \label{uest}
\|u\|^{2}\le\|\Be_{*}\|^{2}+2\|\BG_{*}\|\|v\|+2\|v\|^{2}.
\end{equation}
Applying this inequality to the $\|u\|^{2}$ term in (\ref{penult}) we obtain
\[
\|\BG_{*}\|^{2}\le C_{0}\left(\|\Be_{*}\|^{2}+\|\Be_{*}\|\frac{\|u\|}{h}+\|\BG_{*}\|\|v\|+
\|v\|^{2}\right).
\]
from which the theorem follows.
\end{proof}

\subsection{Proof of Lemma~\ref{lem:KTI} and the Korn inequality}
\label{sub:preK}
  \textbf{Step 1.} First we prove the analog of Lemma~\ref{lem:KTI} in which $\Grad\Bu$ and $e(\Bu)$ are replaced with
\[
\BA=\left[
\begin{array}{ccc}
  u_{r,r} & u_{r,\Gth}-u_{\Gth} & u_{r,z}\\
  u_{\Gth,r} & u_{\Gth,\Gth}+u_{r} & u_{\Gth,z}\\
  u_{z,r} & u_{z,\Gth} & u_{z,z}
\end{array}
\right],\qquad e(\BA)=\hf(\BA+\BA^{T}).
\]
respectively, which is
\begin{equation}
  \label{nor}
\|\BA\|^{2}\le C(L)\|e(\BA)\|\left(\|e(\BA)\|+\frac{\|u_{r}\|}{h}\right).
\end{equation}
To prove inequality (\ref{nor}) we need to estimate three quantities
\[
G_{12}^{2}=\|u_{\Gth,r}\|^{2}+\|u_{r,\Gth}-u_{\Gth}\|^{2},\qquad
G_{13}^{2}=\|u_{r,z}\|^{2}+\|u_{z,r}\|^{2},\qquad
G_{23}^{2}=\|u_{z,\Gth}\|^{2}+\|u_{\Gth,z}\|^{2}
\]
In terms of
\[
E_{12}^{2}=\|u_{\Gth,r}+u_{r,\Gth}-u_{\Gth}\|^{2},\qquad
E_{13}^{2}=\|u_{r,z}+u_{z,r}\|^{2},\qquad
E_{23}^{2}=\|u_{z,\Gth}+u_{\Gth,z}\|^{2},
\]
\[
E_{11}^{2}=G_{11}^{2}=\|u_{r,r}\|^{2},\qquad
E_{22}^{2}=G_{22}^{2}=\|u_{\Gth,\Gth}+u_{r}\|^{2},\qquad
E_{33}^{2}=G_{33}^{2}=\|u_{z,z}\|^{2}.
\]

\textbf{Step 2.} In this step we prove the inequality (\ref{rtheta}) in
Lemma~\ref{lem:KTI}. We start with $G_{23}$. Integration by parts, using the
\bc s $u_{\Gth}=0$ at $z=0$ and $z=L$ and the periodicity in $\Gth$ gives
\[
|(u_{z,\Gth},u_{\Gth,z})|=|(u_{z,z},u_{\Gth,\Gth})|\le\|u_{z,z}\|\|u_{\Gth,\Gth}\|\le
E_{33}(E_{22}+\|u_{r}\|).
\]
Therefore,
\begin{equation}
  \label{G23}
G_{23}^{2}=E_{23}^{2}-2(u_{z,\Gth},u_{\Gth,z})\le
E_{23}^{2}+E_{22}^{2}+E_{33}^{2}+2E_{33}\|u_{r}\|\le2\|e(\BA)\|(\|e(\BA)\|+\|u_{r}\|).
\end{equation}

\textbf{Step 3.} Next we estimate $G_{13}$. Let us fix $\Gth\in[0,2\pi]$ arbitrarily. Next we
apply Theorem~\ref{th:basicineq} to the function
\[
\Bu(r,z)=(u_{r}(r,\Gth,z),u_{z}(r,\Gth,z))
\]
and
$\Ga=0$. We obtain, integrating the
inequality over $\Gth$ as using the Cauchy-Schwartz for the product term
\begin{equation}
  \label{G13}
G_{13}^{2}\le C_{0}\left(E_{11}^{2}+E_{13}^{2}+E_{33}^{2}+
\frac{\|u_{r}\|}{h}(E_{11}+E_{13}+E_{33})\right)\le
C_{0}\|e(\BA)\|\left(\|e(\BA)\|+\frac{\|u_{r}\|}{h}\right),
\end{equation}
where $C_{0}$ is an absolute numerical constant, independent of $h$ and $L$.

\textbf{Step 4.} Finally we estimate $G_{12}$. Let us fix $z\in[0,L]$
arbitrarily. We apply Theorem~\ref{th:hard} to the function
\[
\Bu(r,\Gth)=(u_{r}(r,\Gth,z),u_{\Gth}(r,\Gth,z))
\]
We obtain, integrating the
inequality over $z$ and using the Cauchy-Schwartz for the product term
\[
G_{12}^{2}\le C_{0}\left(\|e(\BA)\|^{2}+\|e(\BA)\|\frac{\|u_{r}\|}{h}+\|u_{\Gth}\|^{2}\right)
\]
We estimate via the 1D Poincar\'e inequality
\begin{equation}
  \label{Poincare}
\|u_{\Gth}\|^{2}\le\frac{L^{2}}{\pi^{2}}\|u_{\Gth,z}\|^{2}\le\frac{L^{2}}{\pi^{2}}G_{23}^{2}
\le \frac{2L^{2}}{\pi^{2}}(\|e(\BA)\|^{2}+\|e(\BA)\|\|u_{r}\|).
\end{equation}
Thus, there exists a constant $C(L)\le C_{0}(L^{2}(\Gs+1)+1)$ such that
\[
G_{12}^{2}\le C(L)\|e(\BA)\|\left(\|e(\BA)\|+\frac{\|u_{r}\|}{h}\right).
\]
Next we prove the analog of the Korn inequality in which again $\Grad\Bu$ and $e(\Bu)$ are replaced with $\BA$ and $e(\BA)$ respectively. Integrating the inequality (\ref{uest}) in $z$ and using Cauchy-Schwartz for
  the product term we obtain
\[
\|u_{r}\|^{2}\le\|e(\BA)\|^{2}+2\|\BA\|\|u_{\Gth}\|+2\|u_{\Gth}\|^{2}\le
\|e(\BA)\|^{2}+\Ge^{2}\|\BA\|^{2}+\left(2+\nth{\Ge^{2}}\right)\|u_{\Gth}\|^{2}
\]
for any $\Ge>0$. The small parameter $\Ge\in(0,1)$ will be chosen later in an
asymptotically optimal way. Applying (\ref{Poincare}) we obtain for
sufficiently small $\Ge$
\[
\|u_{r}\|^{2}\le\left(\frac{L^{2}}{\Ge^{2}}+1\right)\|e(\BA)\|^{2}+\Ge^{2}\|\BA\|^{2}+
\frac{L^{2}}{\Ge^{2}}\|e(\BA)\|\|u_{r}\|.
\]
Therefore,
\[
\|u_{r}\|^{2}\le2\left(\frac{L^{2}}{\Ge^{2}}+1\right)^{2}\|e(\BA)\|^{2}+2\Ge^{2}\|\BA\|^{2}.
\]
Thus,
\[
\|u_{r}\|\le\sqrt{2}\left(\left(\frac{L^{2}}{\Ge^{2}}+1\right)\|e(\BA)\|+\Ge\|\BA\|\right).
\]
Substituting this inequality to (\ref{nor}) we conclude that there is a
constant $C(L)$, depending only on $L$  such that
\[
\|\BA\|^{2}\le C(L)\left(\nth{h\Ge^{2}}+\frac{\Ge^{2}}{h^{2}}\right)\|e(\BA)\|^{2}.
\]
We now choose $\Ge=h^{1/4}$ to minimize the bound:
\begin{equation}
\label{KornA}
\|\BA\|^{2}\le\frac{C(L)}{h\sqrt{h}}\|e(\BA)\|^{2}.
\end{equation}
Theorem~\ref{th:KI} is now an immediate consequence of (\ref{KornA}) and the obvious inequality
\begin{equation}
\label{nabUandA}
\|e(\BU)-e(\BA)\|^2\leq\|\nabla\BU-\BA\|^2\leq h^2\|\BA\|^2.
\end{equation}
Observe that we get from (\ref{nabUandA}) and (\ref{KornA}) that,
\begin{equation}
\label{eUeA}
\|e(\BU)-e(\BA)\|\leq h\|\BA\|\leq C(L) h^{1/4}\|e(\BA)\|.
\end{equation}
Lemma~\ref{lem:KTI} will now follow from (\ref{nabUandA}), (\ref{eUeA}) and (\ref{nor}).


\end{document}